\documentclass[final]{siamonline1116}
\usepackage{showkeys}
\usepackage{amssymb,amsmath}
\usepackage{algorithm}
\usepackage{algorithmic}

\newtheorem{remark}{\bf  Remark}[section]
\newtheorem{assumption}{\bf  Assumption}

\newcommand{\bbC}{{\mathbb{C}}}

\newcommand{\bbR}{{\mathbb{R}}}
\newcommand{\bbN}{{\mathbb{N}}}

\newcommand{\cA}{\mathcal{A}}

\newcommand{\cF}{\mathcal{F}}

\newcommand{\cK}{\mathcal{K}}

\newcommand{\cN}{\mathcal{N}}

\newcommand{\cQ}{\mathcal{Q}}

\newcommand{\cU}{\mathcal{U}}
\newcommand{\cV}{\mathcal{V}}

\newcommand{\argmax}{\operatornamewithlimits{argmax}}

\newcommand{\bsrho}{{\boldsymbol{\rho}}}

\newcommand{\bsnu}{{\boldsymbol{\nu}}}
\newcommand{\bsmu}{{\boldsymbol{\mu}}}

\newcommand{\bse}{{\boldsymbol{e}}}

\newcommand{\bsy}{{\boldsymbol{y}}}
\newcommand{\bsz}{{\boldsymbol{z}}}

\newcommand{\bszero}{{\boldsymbol{0}}}

\newcommand{\beq}{\begin{equation}}
\newcommand{\eeq}{\end{equation}}
\newcommand{\ba}{\begin{array}}
\newcommand{\ea}{\end{array}}

\DeclareMathOperator*{\esssup}{ess\,sup}
\DeclareMathOperator*{\essinf}{ess\,inf}

\title{
Sparse polynomial approximations
for affine parametric saddle point problems 
\thanks{This work was supported by DARPA contract W911NF-15-2-0121, NSF grants CBET-1508713 and ACI-1550593, and DOE grants DE-SC0010518 and DE-SC0009286.}
}

\author{Peng Chen \thanks{Institute for Computational Engineering \&
    Sciences, The University of Texas at Austin, Austin, TX 78712
    (\email{peng@ices.utexas.edu}).}   
    \and Omar Ghattas \thanks{Institute
    for Computational Engineering \& Sciences, Department of
    Mechanical Engineering, and Department of Geological Sciences, The
    University of Texas at Austin, Austin, TX 78712
    (\email{omar@ices.utexas.edu}).}  
    }
\begin{document}
\maketitle
\slugger{sisc}{xxxx}{xx}{x}{x--x}

%
%
%

\begin{abstract}
In this work we study convergence properties of sparse polynomial
approximations for a class of affine parametric saddle point
problems. Such problems can be found in many computational science and
engineering fields, including the Stokes equations for viscous incompressible flow, 
mixed formulation of diffusion equations for heat conduction or ground
water flow, time-harmonic Maxwell equations for electromagnetics,
etc. Due to the lack of knowledge or intrinsic randomness, the
(viscosity, diffusivity, permeability, permittivity, etc.)
coefficients of such problems are uncertain and can often be
represented or approximated by high- or countably infinite-dimensional
random parameters equipped with suitable probability distributions,
and the coefficients affinely depend on a series of either globally or
locally supported basis functions, e.g., Karhunen--Lo\`eve expansion,
piecewise polynomials, or adaptive wavelet
approximations. Consequently, we are faced with solving affine
parametric saddle point problems. Here we study sparse polynomial
approximations of the parametric solutions, in particular sparse
Taylor approximations, and their
convergence properties for these parametric problems.  With suitable
sparsity assumptions on the parametrization, we obtain the algebraic
convergence rates $O(N^{-r})$ for the sparse polynomial approximations
of the parametric solutions, in cases of both globally and locally
supported basis functions. We prove that $r$ depends only on a
sparsity parameter in the parametrization of the random input, and in
particular does not depend on the number of active parameter
dimensions or the number of polynomial terms $N$. These results imply
that sparse polynomial approximations can effectively break the curse
of dimensionality, thereby establishing a theoretical foundation for
the development and application of such practical algorithms as
adaptive, least-squares, and compressive sensing constructions for the
solution of high- or infinite-dimensional affine parametric saddle
point problems.

%
%
%
%
%
%

\end{abstract}



\pagestyle{myheadings}
\thispagestyle{plain}
\markboth{Sparse polynomial approximations for affine parametric saddle point problems}{P. Chen}

\section{Introduction}

Computational simulations based on mathematical models are increasing
used for decision making (design, control, allocation of resources,
determination of policy, etc.). For such cases, it is critical to
account for uncertainties in the inputs, and thus output predictions of
these models.  One fundamental approach to characterize these
uncertainties is by probabilistic modeling, where the uncertain input
can be represented by a finite number of random variables or by random
fields that can be represented by a large or even infinite number of
random variables. We refer to these random variables as parameters and
equip them with suitable probability measures.  With these parameters
as uncertain inputs, we often need to conduct statistical analysis of
the model outputs, such as sensitivity analysis with respect to the
parameters, computation of statistical moments via integration of
outputs in the parameter space, and risk analysis that predicts the
failure probability of the system under the uncertainty. To perform
these statistical analyses, various numerical approximation methods
have been developed largely in the last few decades, such as Monte
Carlo and quasi Monte Carlo methods, generalized polynomial chaos,
stochastic collocation and Galerkin methods, and model and parameter
reduction methods.

The Monte Carlo method has been widely employed in practice because of
several advantages, such as very simple and embarrassingly parallel
implementation and dimension-independent convergence. However, it has
a slow convergence rate of $O(N^{-1/2})$, where $N$ is the number of
samples, requiring a large number of simulations to achieve sufficient
accuracy. New methods such as (high-order) quasi Monte Carlo
\cite{KuoSchwabSloan2012, GantnerSchwab2016} and
multi-level/multi-index Monte Carlo \cite{CliffeGilesScheichlEtAl2011,
  Haji-AliNobileTempone2016} have been proposed to achieve faster
convergence and reduced computational cost. Sparse polynomial
approximations such as stochastic Galerkin and collocation methods
based on (generalized) polynomial chaos and sparse grids have been
developed that improve the convergence to a great extent for problems
depending smoothly on the parameters; see, e.g.,
\cite{XiuKarniadakis2002, GhanemSpanos2003,
  BabuskaTemponeZouraris2004, XiuHesthaven2005,
  BabuskaNobileTempone2007, NobileTemponeWebster2008}. Practical
algorithms to construct such sparse polynomial approximations, such as
adaptive \cite{GerstnerGriebel2003, ChenQuarteroni2015}, least-squares
\cite{ChkifaCohenMiglioratiEtAl2015, NarayanJakemanZhou2017}, and
compressive sensing \cite{DoostanOwhadi2011, RauhutSchwab2017}
constructions, have also been actively developed.  Another class of
methods known as model reduction, including reduced basis methods,
achieve quasi optimal convergence (in terms of Kolmogorov widths
\cite{BinevCohenDahmenEtAl2011})
and considerable computational reduction for many-query simulations \cite{BoyavalLeBrisLelievreEtAl2010, BuffaMadayPateraEtAl2012, BinevCohenDahmenEtAl2011,  ChenQuarteroniRozza2017,  BennerOhlbergerCohenEtAl2017} by exploring the intrinsic structure of the solution manifolds. 

One critical challenge faced by polynomial based approximation methods
for high-dimensional parametric problems is the so-called curse of
dimensionality, i.e., convergence rates that severely deteriorate with
the parameter dimension. In  recent work
\cite{CohenDeVoreSchwab2010, CohenDevoreSchwab2011,
  ChkifaCohenSchwab2015, TranWebsterZhang2017,
  BachmayrCohenMigliorati2017}, it has been demonstrated that the
curse of dimensionality can be effectively broken with
dimension-independent convergence rates achieved under certain
sparsity assumptions on the countably infinite-dimensional
parametrization of the uncertain input.  For instance, in
\cite{CohenDevoreSchwab2011}, analytic regularity of the parametric
solution with respect to the parameters was obtained for elliptic
partial differential equations. This leads to upper bounds for the
coefficients of Taylor expansion of the parametric solution. Under an
$\ell^s$-summability of the basis functions that represent the random
input, the Taylor coefficients were demonstrated to also satisfy the
$\ell^s$-summability. Then a dimension-independent convergence rate of
a sparse Taylor approximation---truncation of a Taylor expansion of
the parametric solution into a suitable sparse index set---were
achieved by Stechkin's lemma. This analysis has been extended to
sparse Legendre polynomial approximation \cite{CohenDevoreSchwab2011},
sparse polynomial interpolation \cite{ChkifaCohenSchwab2014}, and
sparse polynomial integration \cite{SchillingsSchwab2013} for elliptic
problems as well as for certain parabolic and nonlinear problems
\cite{ChkifaCohenSchwab2015}.

In this work, we consider parametric saddle point problems that cover
a wide range of applications, such as the Stokes equations for viscous
incompressible flow, mixed formulation of the Poisson equation for
heat conduction or ground water flow, and time-harmonic Maxwell
equations for electromagnetic wave propagation; see
\cite{Quarteroni2013, BoffiBrezziFortin2013} and references therein.
These applications require better understanding of the approximability
of parametric saddle point problems in a high or infinite dimensional
parametric setting, which is the aim and main contribution of this
work. In particular, our contributions are presented in several
sections structured as follows: In Sec.\ \ref{sec:affineSaddlePoint},
we formulate an abstract saddle point problem with affine
parametrization, and demonstrate the well-posedness of the parametric
saddle point problem through several specific examples. Moreover, we
consider both globally and locally supported basis functions for the
affine parametrization with suitable sparsity assumptions for each of
them. In Sec.\ \ref{sec:sparsePolynomial}, we consider a Taylor
expansion of the solution of the parametric saddle point problem with
respect to the parameters and its sparse Taylor approximation. In the
case of globally supported basis functions, we establish the analytic
regularity of the parametric solution with respect to the parameters,
and prove the $\ell^s$-summability of the Taylor coefficients. In the
case of locally supported basis functions, we prove a weighted
$\ell^2$-summability of the Taylor coefficients, based on which we
obtain the $\ell^s$-summability of the Taylor coefficients. Based on
the $\ell^s$-summability, we prove dimension-independent convergence
rates of the sparse Taylor approximations, for both arbitrary sparse
index set and a downward closed sparse index set.
The last section provides conclusions and several ongoing and future research directions.

\section{Affine parametric saddle point problems}
\label{sec:affineSaddlePoint}
\subsection{An abstract saddle point formulation}
Let $\cV$ and $\cQ$ denote two Hilbert spaces equipped with inner products $(\cdot, \cdot)_\cV$, $(\cdot, \cdot)_\cQ$ and induced norms $||v||_\cV = (v, v)^{1/2}_\cV$, $\forall v \in \cV$, and $||\cdot||^2_\cQ=(q, q)^{1/2}_\cQ$, $\forall q \in \cQ$. Let $\cV'$ and $\cQ'$ denote the duals of $\cV$ and $\cQ$, respectively. Let $\cK$ denote a separable Banach space.
We present an abstract formulation of the parametric saddle point problem as: given parameter $\kappa \in \cK$, and data $f \in \cV'$ and $g\in \cQ'$, find $(u, p) \in \cV \times \cQ$ such that 
\beq\label{eq:saddlePointBase}
\left\{
\begin{aligned}
a(u,v;\kappa) + b(v,p) &= f(v) \quad \forall v \in \cV,\\
b(u,q) & = g(q) \quad \forall q \in \cQ,
\end{aligned}
\right.
\eeq
where the linear forms $f(v)$ and $g(q)$ represent the duality pairing $\langle f, v \rangle_{\cV'\times \cV}$ and $\langle g, q\rangle_{\cQ'\times \cQ}$ for simplicity, $a(\cdot, \cdot; \kappa): \cV\times \cV \to \bbR$ is a parametric bilinear form, and $b(\cdot, \cdot): \cV \times \cQ \to \bbR$ is a bilinear form. Moreover, we make the following assumptions on the bilinear forms. 
First, let $\cV^0$ denote the kernel of the bilinear form $b$ in $\cV$, i.e., 
\beq
\cV^0 :=\{v\in \cV: b(v, q) = 0, \; \forall q \in \cQ\}.
\eeq
\begin{assumption}\label{ass:saddlePointBase}
Suppose the bilinear forms $a(\cdot, \cdot;\kappa)$ and $b(\cdot, \cdot)$ are uniformly continuous, i.e., there exist constants $\gamma> 0$ independent of $\kappa$ and $ \delta > 0$ such that  
\beq\label{eq:continuity}
\begin{split}
a(w, v; \kappa) &\leq \gamma ||w||_\cV ||v||_\cV, \quad \forall w,v\in \cV,\\
 b(v, q) &\leq \delta ||v||_\cV ||q||_\cQ, \quad \forall v \in \cV, q \in \cQ.
\end{split}
\eeq
Moreover, we assume that $a(\cdot, \cdot;\kappa)$ is uniformly coercive in $\cV^0$, i.e., there exists a constant $\alpha > 0$ independent of $\kappa$ such that 
\beq\label{eq:coercivity}
a(v,v;\kappa) \geq \alpha ||v||_\cV^2, \quad \forall v \in \cV^0.
\eeq
Furthermore, we assume that $b(\cdot, \cdot)$ satisfies the inf-sup (compatibility) condition, i.e., there exists a constant $\beta > 0$ such that  
\beq
\inf_{q \in \cQ} \sup_{v \in \cV} \frac{b(v,q)}{||v||_\cV||q||_\cQ} \geq \beta .
\eeq
\end{assumption}
The classical results of existence, uniqueness, and a-priori estimates for the parametric saddle point problem \eqref{eq:saddlePointBase} are stated in the following theorem. 
\begin{theorem}\label{thm:saddlePointBase}
\cite[Theorem 16.4]{Quarteroni2013}.
Under Assumption \ref{ass:saddlePointBase}, for every $\kappa \in \cK$, there exists a unique solution $(u,p) \in \cV \times \cQ$ of the parametric saddle point problem \eqref{eq:saddlePointBase}, such that the following a-priori estimates hold 
\beq\label{eq:stability}
||u||_\cV \leq C_u < \infty \text{ and } ||q||_\cQ \leq C_p < \infty ,
\eeq
where for notational convenience, the constants $C_u$ and $C_p$ are short for 
\beq\label{eq:CuCp}
C_u =\frac{1}{\alpha} ||f||_{\cV'} + \frac{\alpha + \gamma}{\alpha \beta} ||g||_{\cQ'}
\text{ and } C_p = \frac{\alpha + \gamma}{\alpha\beta}  ||f||_{\cV'} + \frac{\gamma(\alpha + \gamma)}{\alpha \beta^2} ||g||_{\cQ'}.
\eeq
\end{theorem}

\subsection{Affine parametrization}
In this section, we present an affine parametrization for the parameter $\kappa$. 
We first present a common structure of the bilinear form $a(\cdot, \cdot;\kappa)$ in \eqref{eq:saddlePointBase} appearing in many saddle point problems such as the Stokes equations, mixed formulation of the Poisson  equation, and time-harmonic Maxwell's equations, that is affine with respect to the parameter $\kappa \in \cK$, i.e., it can be written as
\beq\label{eq:affineA}
a(w, v; \kappa) = a_0(w, v) + a_1(w,v;\kappa), \quad \forall w, v \in \cV,
\eeq
where $a_1(w, v;\kappa)$ depends linearly on $\kappa$ such that for any $\kappa \in \cK$ there hold
\beq\label{eq:bounda1}
\begin{split}
a_1(v,v;\kappa) &\geq c_1 \essinf_{x \in D} |\kappa(x)|\, ||v||^2_\cV, \quad \forall v \in \cV^0,  \\
a_1(w,v;\kappa) &\leq C_1 \esssup_{x \in D} |\kappa(x)|\, ||w||_\cV ||v||_\cV, \quad \forall w, v \in \cV,\\
a_1(w,v;\kappa) &\leq \frac{1}{2} \left(a_1(w,w; |\kappa|) + a_1(v,v;|\kappa|)\right), \quad \forall w, v \in \cV.
\end{split}
\eeq
for constants $c_1, C_1 > 0$ independent of $\kappa$, e.g., related to the Poincar\'e's or Friedrichs' constant in Stokes equations or time-harmonic Maxwell's equations.
We shall consider this affine structure \eqref{eq:affineA} with the properties \eqref{eq:bounda1} in what follows. 

To parametrize $\kappa$, we consider a countably infinite-dimensional parameter space
\beq
U = [-1,1]^\bbN.
\eeq
We denote the element of the parameter space as $\bsy = (y_j)_{j\geq 1} \in U$ and equip the parameter space with the probability measure 
\beq
d\mu(\bsy) = \bigotimes_{j\geq 1} \frac{d\lambda(y_j)}{2} ,
\eeq
where $d\lambda$ is the Lebesgue measure in $[-1,1]$.
To this end, we consider an affine parametrization for the representation and approximation of the parameter $\kappa$ that is widely used in the literature \cite{BabuskaTemponeZouraris2004, BabuskaNobileTempone2007, CohenDeVoreSchwab2010, CohenDevoreSchwab2011, CliffeGilesScheichlEtAl2011, ChkifaCohenDeVoreEtAl2013, CohenDeVore2015, Haji-AliNobileTempone2016, Soize2017}. 
\begin{assumption} 
\label{ass:UEA}
The variation of the parameter $\kappa$ in $\cK$ can be represented by the parameter $\bsy \in U$ through the affine expansion
\beq\label{eq:affinePara}
\kappa(x,\bsy) = \kappa_0(x) + \sum_{j\geq 1} y_j \kappa_j(x), \; \forall (x,\bsy) \in D \times U, \text{ and } \kappa_j \in \cK, \; \forall j \geq 0.
\eeq
Moreover, 
we assume there exist constants $0 < \theta < \Theta < \infty $ such that 
\beq
\theta < \kappa_{\text{min}} := \inf_{(x,\bsy) \in D\times U}\kappa(x,\bsy) \leq \sup_{(x,\bsy) \in D\times U} \kappa(x,\bsy) =: \kappa_{\text{max}} < \frac{\Theta}{2},
\eeq
and such that the coercivity and continuity conditions \eqref{eq:coercivity} and \eqref{eq:continuity} are satisfied for for the bilinear form $a(\cdot, \cdot; \kappa)$ at any $\kappa \in [\theta, \Theta]$. 
\end{assumption}

The sequence $(\kappa_j)_{j\geq 0}$ could either be directly 
prescribed knowledge of the physical system 
or given by an affine representation or approximation of the random field $\kappa$. 
We present two specific examples, where we distinguish the parametrization in two classes representing globally and locally supported basis $(\kappa_j)_{j\geq 1}$, respectively.
\begin{enumerate}
\item \emph{Globally supported basis}.
One classical example comes from Karhunen--Lo\`eve expansion of a random field with finite second order moment, given by \cite{SchwabTodor2006}
\beq\label{eq:KL}
\kappa(x, \bsy) = \kappa_0(x) + \sum_{j = 1}^\infty y_j \sqrt{\lambda_j} \psi_j(x),
\eeq
where $\kappa_0$ is the mean of the random field, $(\lambda_j, \psi_j)_{j\geq 1}$ are the eigenpairs of the covariance of the random field. Here, we can identify $\kappa_j = \sqrt{\lambda_j} \psi_j$, $j \geq 1$, in the affine assumption \eqref{eq:affinePara}. 

\item \emph{Locally supported basis}. Piecewise polynomials or wavelets can be employed to model or approximate the parameter field $\kappa$. A particular case is the weighted piecewise constant basis representation
\beq\label{eq:localExpansion}
\kappa(x, \bsy) = \kappa_0 + \sum_{j = 1}^J y_j w_j \chi_j(x),
\eeq
where $w_j$ is the weight and $\chi_j$ is the characteristic function in the subdomain/element $D_j$, $j = 1, \dots, J$, where $D = \cup_{j=1}^J D_j$ and $D_i \cap D_j = \emptyset$ for $i \neq j$. In this example, we can identify $\kappa_j = w_j \chi_j$, $j = 1, \dots, J$. 
\end{enumerate}


Assumption \ref{ass:UEA} guarantees the well-posedness of the parametric saddle point problem \eqref{eq:saddlePointBase}. 
To study the convergence property of certain approximation of its solution or related quantity of interest, we make the following assumptions to cover the globally and locally supported basis representations, which appear, e.g., in \cite{CohenDevoreSchwab2011} and \cite{BachmayrCohenMigliorati2017}.

\begin{assumption}\label{ass:lpSum}
For the parametrization \eqref{eq:affinePara} under Assumption \ref{ass:UEA}, assume for some $s \in (0, 1)$ there holds $(||\kappa_j||_{\cK})_{j\geq 1} \in \ell^s(\bbN)$, i.e.,
\beq\label{eq:lpSum}
\sum_{j\geq 1} ||\kappa_j ||_{\cK}^s < \infty.
\eeq
\end{assumption}

\begin{remark}
\label{remark:KLtoLocal}
For the Karhunen--Lo\`eve expansion \eqref{eq:KL}, the $\ell^s$-summability condition \eqref{eq:lpSum} is satisfied when $\sup_{j\geq 1}||\psi_j||_{\cK} \leq C$ for some $C < \infty$, and $(\sqrt{\lambda_j})_{j\geq 1} \in \ell^s(\bbN)$. However, it is not satisfied for any $s \in (0, 1)$ in the case of the locally supported representation \eqref{eq:localExpansion} when $|w_j| \lnsim j^{-1}$, i.e., $(|w_j|)_{j\geq 1} \not\in \ell^1(\bbN)$, as $J \to \infty$. To accommodate such a case, we make the following assumption.
\end{remark}

\begin{assumption}\label{ass:rhoSum}
For the parametrization \eqref{eq:affinePara} under Assumption \ref{ass:UEA}, assume there exists a sequence $\bsrho = (\rho_j)_{j\geq 1}$ with $\rho_j > 1$, such that 
\beq\label{eq:rhoSum}
\sum_{j\geq 1} \rho_j |\kappa_j(x)| \leq \kappa_0(x) - \epsilon, \quad \forall x \in D,
\eeq
for some $\theta< \epsilon < \kappa_{\text{min}}$, and such that $(\rho_j^{-1})_{j\geq 1} \in \ell^t(\bbN)$ for some $t \in (0, \infty)$.
\end{assumption}

\begin{remark}\label{remark:lpSum}
We can see that Assumption \ref{ass:rhoSum} is satisfied for the locally supported representation \eqref{eq:localExpansion} for $J \to \infty$. 
For instance, we can take $\rho_j^{-1} \sim |w_j| $ and $\rho_j |w_j| \leq \kappa_{\text{min}} - \epsilon $ as $|w_j| \to 0$, 
such that $\rho_j > 1$ and \eqref{eq:rhoSum} holds, 
then $(\rho_j^{-1})_{j\geq 1} \in \ell^t(\bbN)$ whenever $(|w_j|)_{j\geq 1} \in \ell^t(\bbN)$ for any $t\in (0,\infty)$.
\end{remark}

\section{Sparse polynomial approximations}
\label{sec:sparsePolynomial}
Let $\cF$ denote a multi-index set with finitely supported multi-index $\bsnu = (\nu_j)_{j\geq 1}$, i.e., $\bsnu \in \cF$ if and only if $|\bsnu| = \sum_{j\geq 1} \nu_j < \infty $.
For any $\bsnu \in \cF$,  we define the multi-factorial $\bsnu !$, multi-monomial $\bsy^\bsnu$  for $\bsy \in U$, and partial derivative $\partial^\bsnu \psi(\bsy)$ for a differentiable parametric map $\psi(\bsy)$ as 
\beq
\quad \bsnu! := \prod_{j\geq 1} \nu_j !, \quad \bsy^\bsnu :=  \prod_{j\geq 1} y_j^{\nu_j}, \quad \partial^\bsnu \psi(\bsy) := \frac{\partial^{|\bsnu|} \psi(\bsy)}{\partial^{\nu_1} y_1 \partial^{\nu_2} y_2\cdots },
\eeq
where we use the convention $0!: = 1$, $0^0:=1$, and $\partial^0 \psi(\bsy)/\partial^0 y_j = \psi(\bsy)$. 
For such a differentiable map $\psi$, we consider the Taylor power series 
\beq\label{eq:TaylorSeries}
T_{\cF} \psi(\bsy) :=\sum_{\bsnu \in \cF} t^\psi_\bsnu \bsy^\bsnu,
\eeq
with Taylor coefficients $t_\bsnu^\psi$ defined as 
\beq
t_\bsnu^\psi := \frac{1}{\bsnu!} \partial^\bsnu \psi(\bszero), \quad \bsnu \in \cF.
\eeq
Let $(\Lambda_N)_{N\geq 1} \subset \cF$ denote a sequence of  index sets with $N$ indices that exhaust $\cF$, i.e., any finite set $\Lambda \subset \cF$ is contained in all $\Lambda_N$ for $N \geq N_0$ with $N_0$ sufficiently large. We define the truncation of the power series \eqref{eq:TaylorSeries} in $\Lambda_N$ as 
\beq
T_{\Lambda_N} \psi(\bsy) : = \sum_{\bsnu \in \Lambda_N} t^\psi_\bsnu(\bsy) \bsy^\bsnu, 
\eeq
which we call \emph{sparse Taylor approximation}. We are interested in two questions: (1) if the sparse Taylor approximation for the solution of the parametric saddle point problem \eqref{eq:saddlePointBase} is convergent; (2) if so, how fast it converges with respect to $N$. To answer these questions, we carry out two types of analyses corresponding to Assumption \ref{ass:lpSum} and \ref{ass:rhoSum}, respectively. The first type is to obtain the analytic regularity property of the parametric solution in a complex domain covering the parameter space. This analyticity leads to upper bounds for the Taylor coefficients $(t^u_\bsnu, t_\bsnu^p)$ at each $\bsnu \in \cF$ by Cauchy's integral formula, which implies a $\ell^s(\cF)$-summability of the coefficients. The second type is to derive a weighted $\ell^2(\cF)$-summability of the Taylor coefficient based on the affine structure of the parametrization; then the $\ell^s(\cF)$-summability of the Taylor coefficients is obtained by using H\"older's inequality. Due to the $\ell^s(\cF)$-summability, a best $N$-term dimension-independent convergence rate of a suitable Taylor approximation is achieved using Stechkin's lemma. These analyses are based on the results in \cite{CohenDevoreSchwab2011} and \cite{BachmayrCohenMigliorati2017} for studying parametric elliptic PDEs, which we extend to dealing with the parametric saddle point problem \eqref{eq:saddlePointBase} under Assumption \ref{ass:lpSum} and \ref{ass:rhoSum}, respectively.

\subsection{$\ell^s$-summability by analytic regularity}
\label{sec:analyticRegularity}
Let $\bsz = (z_j)_{j\geq 1}$ denote a sequence of complex numbers with $z_j \in \bbC$, $j\geq 1$, i.e., $\bsz \in \bbC^\bbN$. 
Let $\cU$ denote a polydisc defined as 
\beq
\cU := \left\{\bsz \in \bbC^\bbN:  |z_j| \leq 1 \text{ for every } j \geq 1\right\}. 
\eeq
Then we can extend the parametrization of $\kappa$ in \eqref{eq:affinePara} from $U = [-1,1]^\bbN$ to $\cU$, i.e., 
\beq\label{eq:affineParaC}
\kappa(x,\bsz) = \kappa_0(x) + \sum_{j\geq 1} z_j \kappa_j(x), \quad \forall (x,\bsz) \in D \times \cU,
\eeq
for which, under Assumption \ref{ass:UEA}, we have
\beq
\kappa_{\text{min}} \leq \Re(\kappa(x, \bsz)) \leq |\kappa(x, \bsz)| \leq 2\kappa_{\text{max}}.
\eeq
For two constants $r$ and $R$ such that 
\beq
0< \theta < r < \kappa_{\text{min}} < 2\kappa_{\text{max}} < R < \Theta < \infty,
\eeq 
where $\theta$ and $\Theta$ are given in Assumption \ref{ass:UEA}, we define the complex set 
\beq
\cA_r^R = \{\bsz \in \bbC^\bbN: r \leq  |\kappa(x, \bsz)| \leq R \text{ for every } x \in D\}.
\eeq
Then Theorem \ref{thm:saddlePointBase} holds for $\bsz \in \cA_r^R$ under Assumption \ref{ass:saddlePointBase} and \ref{ass:UEA}, i.e., there exists a unique solution $(u(\bsz), p(\bsz)) \in \cV \times \cQ$, $ \forall \bsz \in \cA_r^R$, which satisfies the a-priori estimates in \eqref{eq:stability}. In fact, Theorem \ref{thm:saddlePointBase} holds for $\bsz \in \cA_{\tilde{r}}^R$ for any $\tilde{r} \geq \theta$ due to Assumption \ref{ass:UEA} on the coercivity condition of the bilinear form $a(\cdot, \cdot; \kappa)$. Moreover, we observe that $\cU \in \cA_r^R$ by definition so that Theorem \ref{thm:saddlePointBase} also holds for $\bsz \in \cU$. 

\begin{lemma}\label{lemma:pertubation}
Let $(u, p)$ and $(\tilde{u}, \tilde{p})$ denote the solutions of the parametric saddle point problem \eqref{eq:saddlePointBase} at $\kappa \in \cA_r^R$ and $\tilde{\kappa} \in \cA_r^R$, respectively, then we have 
\beq\label{eq:pertubation}
||u - \tilde{u}||_\cV \leq \frac{1}{\alpha} C_1C_u ||\kappa - \tilde{\kappa}||_{\cK} \text{ and }
||p - \tilde{p}||_\cQ \leq \frac{\alpha + \gamma}{\alpha + \beta} C_1 C_u ||\kappa - \tilde{\kappa}||_{\cK},
\eeq
where the constants $\alpha$, $\beta$  and $\gamma$ are given in Theorem \ref{thm:saddlePointBase}, $C_1$ and $C_u$ are given in \eqref{eq:bounda1} and \eqref{eq:CuCp}.
\end{lemma}
\begin{proof}
By subtracting \eqref{eq:saddlePointBase} at $\kappa$ from it at $\tilde{\kappa}$, we have 
\beq\label{eq:saddlePointBaseTilde}
\left\{
\begin{aligned}
a(u-\tilde{u},v;\kappa) + b(v,p-\tilde{p}) &= -a(\tilde{u}, v; \kappa - \tilde{\kappa}) \quad \forall v \in \cV,\\
b(u-\tilde{u},q) & = 0 \quad \forall q \in \cQ.
\end{aligned}
\right.
\eeq
By Theorem \ref{thm:saddlePointBase}, the following a-priori estimates hold 
\beq\label{eq:boundtidle}
||u - \tilde{u}||_\cV \leq \frac{1}{\alpha} ||\mathrm{a}||_{\cV'}
\text{ and } ||p - \tilde{p}||_\cQ \leq \frac{\alpha + \gamma}{\alpha + \beta} ||\mathrm{a}||_{\cV'},
\eeq
where we denote $\mathrm{a}(v) = -a(\tilde{u};v;\kappa-\tilde{\kappa})$, $\forall v \in \cV$. By the affine dependence of $a(\cdot, \cdot;\kappa)$ on $\kappa$ as in \eqref{eq:affineA} and the bound \eqref{eq:bounda1} and \eqref{eq:stability}, we have 
\beq
||\mathrm{a}||_{\cV'} \leq C_1 ||\tilde{u}||_\cV ||\kappa - \tilde{\kappa}||_{\cK}  \leq C_1 C_u ||\kappa - \tilde{\kappa}||_{\cK} .
\eeq
Thus, we conclude by inserting this bound in \eqref{eq:boundtidle}.
\end{proof}

\begin{lemma}\label{lemma:complexDeri}
For every $\bsz \in \cA_r^R$, the complex derivative $(\partial_{z_j}u(\bsz), \partial_{z_j} p(\bsz))$ with respect to $z_j$ for each $j \geq 1$ is well-defined for the solution $(u(\bsz), p(\bsz))$ of the parametric saddle point problem \eqref{eq:saddlePointBase}, which is given by: find $(\partial_{z_j}u(\bsz), \partial_{z_j} p(\bsz)) \in \cV \times \cQ$ such that 
\beq\label{eq:saddlePointDeri}
\left\{
\begin{aligned}
a(\partial_{z_j}u,v;\kappa) + b(v,\partial_{z_j}p) &= - a(u, v; \kappa_j) \quad \forall v \in \cV,\\
b(\partial_{z_j} u,q) & = 0 \quad \forall q \in \cQ.
\end{aligned}
\right.
\eeq 
\end{lemma}
Note that we use $a(u, v; \kappa_j) = \int_D \kappa_j (\nabla \times u) \cdot (\nabla \times v) dx$ by slight abuse of notation for the time harmonic Maxwell system, which is bounded.

\begin{proof}
For any $\bsz \in \cA_r^R$ and $j\geq 1$, for $h \in \bbC \setminus \{0\}$ sufficiently small such that $|h| ||\kappa_j||_{\cK} \leq \epsilon < r$, we have 
\beq
r - \epsilon \leq \Re(\kappa(x, \bsz + h\bse_j)) \leq |\kappa(x, \bsz + h\bse_j)| \leq R + \epsilon, \quad \forall x \in D, 
\eeq
where $\bse_j$ is the Kronecker sequence with $1$ at index $j$ and $0$ at other indices,
so that $(u(\bsz + h\bse_j), p(\bsz + h\bse_j)) \in \cV \times \cQ$ is a well-defined solution of \eqref{eq:saddlePointBase} at $\kappa(\bsz + h\bse_j)$. 
Therefore, we have that the following difference quotients satisfy
\beq
u_h(\bsz) := \frac{u(\bsz + h\bse_j) - u(\bsz)}{h} \in \cV \text{ and } p_h(\bsz) := \frac{p(\bsz + h\bse_j) - p(\bsz)}{h} \in \cQ.
\eeq
Subtracting problem \eqref{eq:saddlePointBase} at
$\kappa(\bsz+h\bse_j)$ from its evaluation at $\kappa(\bsz)$ and
dividing by $h$, we obtain that $(u_h(\bsz), p_h(\bsz))$ is a unique
solution of the following problem:  
\beq\label{eq:saddlePointDiffQuot}
\left\{
\begin{aligned}
a(u_h(\bsz),v;\kappa(\bsz)) + b(v,p_h(\bsz)) &= - a(u(\bsz + h\bse_j), v; \kappa_j) \quad \forall v \in \cV,\\
b(u_h(\bsz),q) & = 0 \quad \forall q \in \cQ.
\end{aligned}
\right.
\eeq 
Let $\mathrm{a}_h(v) = - a(u(\bsz + h\bse_j), v; \kappa_j)$. By Assumption \ref{ass:saddlePointBase}, we have 
\beq
|\mathrm{a}_h(v) - \mathrm{a}_0(v)| \leq \gamma ||u(\bsz + h\bse_j) - u(\bsz)||_\cV ||v||_\cV.
\eeq
By the stability estimates \eqref{eq:pertubation} in Lemma \ref{lemma:pertubation}, we have 
\beq
 ||u(\bsz + h\bse_j) - u(\bsz)||_\cV \leq\frac{1}{\alpha}C_1C_u ||\kappa_j||_{\cK} |h|,
\eeq
which converges to zero as $|h| \to 0$, so that $\mathrm{a}_h \to \mathrm{a}_0$ in $\cV'$ as $|h| \to 0$. Consequently, $(u_h, p_h)$ converges to $(u_0, p_0)$ in $\cV\times \cQ$ by Theorem \ref{thm:saddlePointBase}, which is the unique solution of \eqref{eq:saddlePointDiffQuot} for $h = 0$. Therefore, $(\partial_{z_j}u, \partial_{z_j}p) = (u_0, p_0)$ by the uniqueness.
\end{proof}

To study the convergence rate of the Taylor approximation, we need to bound the Taylor coefficients under Assumption \ref{ass:lpSum}, for which we employ the Cauchy integral formula in a suitable complex domain. We call a sequence $\bsrho = (\rho_j)_{j\geq 1}$ is $r$-admissible 
\beq\label{eq:rhor}
\text{ if } \sum_{j\geq 1} \rho_j |\kappa_j(x)| \leq \kappa_0(x) - r\text{ and } \rho_j > 1 \text{ for every } j \geq 1 .
\eeq
By this definition, if $\bsrho$ is $r$-admissible, Theorem \ref{thm:saddlePointBase} holds in a larger polydisc 
\beq
\cU_\bsrho :=\left\{ \bsz \in \bbC^\bbN: |z_j| \leq \rho_j \text{ for every } j \geq 1\right\}.
\eeq
This is because $\cU_\bsrho \subset \cA_r^R$, as it can be readily shown that 
\beq
|\kappa(x,\bsz)| \geq \kappa_0(x) - \sum_{j\geq 1} \rho_j|\kappa_j(x)| \geq r, 
\eeq
and 
\beq
|\kappa(x,\bsz)| \leq \kappa_0(x) + \sum_{j\geq 1} \rho_j |\kappa_j(x)| \leq 2 \kappa_0(x) -r < R.
\eeq

\begin{lemma}
\label{lemma:TaylorCoeffBound}
Under Assumption \ref{ass:saddlePointBase} and \ref{ass:UEA}, for a sequence $\bsrho$ satisfying \eqref{eq:rhor}, for the Taylor coefficients $t_\bsnu^u$ and $t_\bsnu^p$ defined in \eqref{eq:TaylorSeries} we have the following bounds 
\beq\label{eq:TaylorEst}
||t_\bsnu^u||_\cV \leq C_u \bsrho^{-\bsnu} \text{ and } ||t_\bsnu^p||_\cQ \leq C_p \bsrho^{-\bsnu}, \quad \forall \bsnu \in \cF,
\eeq
where $C_u$ and $C_p$ are given in \eqref{eq:CuCp}, $\rho^{-0} = 1$ by convention for any $\rho > 0$. 
\end{lemma}

\begin{proof}
For any $\bsnu \in \cF$, let $J = \max\{j \in \bbN: \nu_j \neq 0\}$. For such $J$, let $\bsz_J^0$ denote a truncated complex sequence for any $\bsz \in \cU$ defined as 
\beq\label{eq:zJ0}
(\bsz_J^0)_j = z_j, \text{ for } 1 \leq j \leq J \text{ and } (\bsz_J^0)_j = 0, \text{ for } j > J.
\eeq 
Then for the solution $(u, p)$ of \eqref{eq:saddlePointBase} at $\bsz_J^0$, we have the a-priori estimates \eqref{eq:stability} by Theorem \ref{thm:saddlePointBase} under Assumption \ref{ass:saddlePointBase} and \ref{ass:UEA}. 
Given the sequence $\bsrho$, we define a new sequence $\tilde{\bsrho}$ as 
\beq
\tilde{\rho}_j = \rho_j + \varepsilon, \text{ if } j\geq J \text{ and } \tilde{\rho}_j = \rho_j, \text{ if } j > J, \varepsilon := \frac{r-\theta}{2||\sum_{1\leq j \leq J} |\kappa_j| ||_{\cK}},
\eeq
which implies $\cU_{\tilde{\bsrho}} \subset \cA_{\tilde{r}}^R$ with $\tilde{r} = (r+\theta)/2 > \theta$. As the coercivity condition \eqref{eq:coercivity} is satisfied for any $\bsz \in \cA_{\tilde{r}}^R$ under Assumption \ref{ass:UEA},  Theorem \ref{thm:saddlePointBase} and Lemma \ref{lemma:complexDeri} hold. Therefore, $u(\bsz_J^0)$ is analytic with respect to each $z_j$, $1\leq j \leq J$ on the polydisc $\cU_{\tilde{\bsrho},J}$, which is an open neighborhood of $\cU_{\bsrho,J}$ defined as 
\beq
\cU_{\bsrho,J} = \left\{(z_1, \dots, z_J) \in \bbC^J: |z_j| \leq \rho_j, \text{ for every } 1\leq j \leq J\right\}.
\eeq
Therefore, by the Cauchy integral formula \cite[Theorem 2.1.2]{Herve1989}, we have for $u$
\beq
u(\tilde{\bsz}_J^0) = (2\pi i)^{-J} \int_{|z_1|=\rho_1} \cdots \int_{|z_J|=\rho_J} \frac{u(\bsz_J^0)}{(\tilde{z}_1-z_1)\cdots (\tilde{z}_J - z_J)} dz_1 \cdots dz_J.
\eeq
By taking the derivative $\partial^\bsnu$ on both sides and evaluating it at $\bszero$, we have 
\beq
\partial^\bsnu u(\bszero) = \bsnu! (2\pi i)^{-J} \int_{|z_1|=\rho_1} \cdots \int_{|z_J|=\rho_J} \frac{u(\bsz_J^0)}{z_1^{\nu_1}\cdots  z_J^{\nu_J}} dz_1 \cdots dz_J,
\eeq
so that 
\beq
\frac{1}{\bsnu!} ||\partial^{\bsnu} u(\bszero)||_\cV \leq \sup_{\bsz_J^0 \in \cU_{\bsrho}}||u(\bsz_J^0)||_\cV \prod_{1\leq j \leq J} \rho_j^{-\nu_j} \leq C_u \bsrho^{-\bsnu},
\eeq
which is \eqref{eq:TaylorEst} for $u$. The same argument is applied to derive the bound for $p$.
\end{proof}

\begin{lemma}
\label{lemma:ParaTaylorSum}
Under Assumption \ref{ass:lpSum}, there exists a $\frac{r+\theta}{2}$-admissible sequence $\bsrho$, i.e, it satisfies \eqref{eq:rhor} with $r$ replaced by $\frac{r+\theta}{2}$, such that 
\beq
\sum_{\bsnu \in \cF} ||t^u_\bsnu||_\cV^s < \infty \text{ and } \sum_{\bsnu \in \cF} ||t^p_\bsnu||_\cQ^s < \infty.
\eeq
\end{lemma}

\begin{proof}
By Lemma \eqref{lemma:TaylorCoeffBound}, we only need to prove there exists a $\frac{r+\theta}{2}$-admissible sequence $\bsrho$ such that 
\beq\label{eq:srho}
\sum_{\bsnu \in \cF} \bsrho^{-s\bsnu} < \infty. 
\eeq
This is done in a constructive way by specification of $\bsrho$. 
By Assumption \ref{ass:lpSum}, we have $(||\kappa_j||_{\cK})_{j\geq 1} \in \ell^s(\bbN) \subset \ell^1(\bbN)$, so that there exists a sufficiently large $J$ such that 
\beq
\sum_{j > J} ||\kappa_j||_{\cK} \leq \frac{r-\theta}{12}.
\eeq
Then we choose $\tau > 1$ such that 
\beq
(\tau-1) \sum_{j \leq J} ||\kappa_j||_{\cK} \leq \frac{r - \theta}{4}.
\eeq
For any $\bsnu \in \cF$, we specify the sequence $\bsrho$ as 
\beq\label{eq:rhoj}
\rho_j := \tau, \; j \leq J; \quad \rho_j := \max\left\{1, \frac{(r-\theta) \nu_j}{4 ||\kappa_j||_{\cK}\sum_{i>J}\nu_i}\right\}, \; j > J,
\eeq
with the convention that $\nu_j/(\sum_{i > J} \nu_i) = 0$ if $\sum_{i > J} \nu_i = 0$. Then we have 
\beq
\sum_{j \geq 1} \rho_j |\kappa_j(x)| \leq \sum_{j \geq 1} |\kappa_j(x)| + \frac{r-\theta}{2} \leq \kappa_0(x) - \frac{r+\theta}{2},
\eeq
where in the second inequality we have used Assumption \ref{ass:UEA}, i.e., for any $x \in D$,
\beq
r < \kappa_0(x) + \inf_{\bsy \in U} \sum_{j\geq 1} y_j \kappa_j(x) = \kappa_0(x) - \sum_{j\geq 1} |\kappa_j(x)|.
\eeq
Therefore, $\bsrho$ is $\frac{r+\theta}{2}$-admissible. By results in \cite[Sec.\ 3]{CohenDevoreSchwab2011}, \eqref{eq:srho} holds for the choice \eqref{eq:rhoj}. 
\end{proof}

\subsection{$\ell^s$-summability by weighted $\ell^2$-summability}
\label{sec:weightedL2}
The $\ell^s$-summability of the Taylor coefficients is guaranteed by the $\ell^s$-summability of $(||\kappa_j||_{\cK})_{j\geq 1}$ in Assumption \ref{ass:lpSum} as shown in the last section. However, as indicated in Remark \ref{remark:lpSum}, $(||\kappa_j||_{\cK})_{j\geq 1}$ may not be $\ell^s$-summable for any $s \in (0, 1)$, as considered in \cite{BachmayrCohenMigliorati2017} for coercive elliptic PDEs. In this case, Assumption \ref{ass:rhoSum} may still hold, in particular for locally supported $(\kappa_j)_{j\geq 1}$, for which we prove the $\ell^s$-summability of the Taylor coefficients $(||t_\bsnu^u||_\cV)_{\bsnu \in \cF}$ and the $\ell^t$-summability of the Taylor coefficients $(||t_\bsnu^p||_\cQ)_{\bsnu \in \cF}$, where $s = \frac{2t}{2+t}$ for $t \in (0, \infty)$ given in Assumption \ref{ass:rhoSum}. 

\begin{lemma}\label{lemma:wLsL2}
Under Assumption \ref{ass:rhoSum}, we have 
\beq
\sum_{\bsnu \in \cF} ||t_\bsnu^u||_\cV^s < \infty, \text{ and } \sum_{\bsnu \in \cF} ||t_\bsnu^p||_\cQ^t < \infty,
\eeq
where $s = \frac{2t}{2+t} \in (0, 2)$ for $t \in (0, \infty)$ given in Assumption \ref{ass:rhoSum}.
\end{lemma}

\begin{proof}
For a sequence $\bsrho$ satisfying \eqref{eq:rhoSum} in Assumption \ref{ass:rhoSum}, 
we define the scaling function $R_\bsrho(\bsy) := (\rho_j y_j)_{j\geq 1}$. By assumption \eqref{eq:rhoSum} we have for any $x \in D$
\beq
\inf_{\bsy \in U} \kappa(x, R_\bsrho(\bsy)) = \kappa_0(x) + \inf_{\bsy \in U}\sum_{j \geq 1} \rho_j y_j \kappa_j(x) \geq \kappa_0(x) - \sum_{j\geq 1}\rho_j |\kappa_j(x)|  \geq \epsilon > \theta,
\eeq
so that the bilinear form $a(\cdot, \cdot; \kappa)$ is coercive by Assumption \ref{ass:UEA}. Under Assumption \ref{ass:saddlePointBase}, there exists a unique $(u(R_\bsrho(\bsy)), p(R_\bsrho(\bsy))) \in \cV \times \cQ$ for every $\bsy \in U$ such that 
\beq\label{eq:saddlePointBaseRrho}
\left\{
\begin{aligned}
a(u(R_\bsrho(\bsy))),v;\kappa(R_\bsrho(\bsy))) + b(v,p(R_\bsrho(\bsy)))) &= f(v) \quad \forall v \in \cV,\\
b(u(R_\bsrho(\bsy))),q) & = g(q) \quad \forall q \in \cQ.
\end{aligned}
\right.
\eeq
By the definition of the Taylor coefficients in \eqref{eq:TaylorSeries}, we have at $\bsnu = \bszero$ that $(t_\bszero^u, t_\bszero^p) = (u(\bszero), p(\bszero))$, which satisfy the a-priori estimates \eqref{eq:stability} by Theorem \ref{thm:saddlePointBase}, i.e., 
\beq
||t_\bszero^u||_\cV \leq C_u \text{ and } ||t_\bszero^p||_\cQ \leq C_p. 
\eeq
For any other $\bsnu \in \cF$, 
by taking the partial derivative $\partial^\bsnu$ for \eqref{eq:saddlePointBaseRrho}, we obtain 
\beq\label{eq:saddlePointBaseRDrho}
\left\{
\begin{aligned}
a(\bsrho^\bsnu \partial^\bsnu u(R_\bsrho(\bsy))),v;\kappa(R_\bsrho(\bsy))) + b(v, \bsrho^\bsnu \partial^\bsnu p(R_\bsrho(\bsy)))) \\
= - \sum_{j \in \text{supp}\, \bsnu} a_1(\nu_j \bsrho^{\bsnu-\bse_j} \partial^{\bsnu-\bse_j} u(R_\bsrho(\bsy))),v; \rho_j \kappa_j) 
\quad \forall v \in \cV,\\
b(\bsrho^\bsnu \partial^\bsnu  u(R_\bsrho(\bsy))),q)  = 0 \quad \forall q \in \cQ,
\end{aligned}
\right.
\eeq
where $\text{supp}\, \bsnu = \{j \in \bbN: \nu_j \neq 0\}$. 
Taking division by $\bsnu!$ on both sides, setting $\bsy = \bszero$, we have the saddle point problem for the Taylor coefficients $(t_\bsnu^u, t_\bsnu^p) \in \cV\times \cQ$
\beq\label{eq:saddlePointBaseTaylorCoeff}
\left\{
\begin{aligned}
a(\bsrho^\bsnu t_\bsnu^u,v;\kappa(R_\bsrho(\bsy))) + b(v, \bsrho^\bsnu t_\bsnu^p) \\
= - \sum_{j \in \text{supp}\, \bsnu} a_1(\bsrho^{\bsnu-\bse_j} t_{\bsnu-\bse_j}^u,v; \rho_j \kappa_j) 
\quad \forall v \in \cV,\\
b(\bsrho^\bsnu t_\bsnu^u, q)  = 0 \quad \forall q \in \cQ.
\end{aligned}
\right.
\eeq
Therefore, $t_\bsnu^u \in \cV^0$ by the second equation. We shall show that $(\bsrho^\bsnu t_\bsnu^u, \bsrho^\bsnu t_\bsnu^p) \in \cV\times \cQ$ is a bounded solution of \eqref{eq:saddlePointBaseTaylorCoeff} for any $\bsnu \in \cF$.
First it is so for $\bsnu = \bszero$. Then by induction we assume that $(\bsrho^\bsmu t_\bsmu^u, \bsrho^\bsmu t_\bsmu^p) \in \cV \times \cQ$ are bounded solutions of \eqref{eq:saddlePointBaseTaylorCoeff} (being $\bsnu$ replaced by $\bsmu$) for any $\bsmu \preceq \bsnu$, i.e., $\mu_j \leq \nu_j$, $\forall j \geq 1$, and $\bsmu \neq \bsnu$, then by Theorem \ref{thm:saddlePointBase} we have 
$(\bsrho^\bsnu t_\bsnu^u, \bsrho^\bsnu t_\bsnu^p) \in \cV\times \cQ$ is the unique solution of 
\eqref{eq:saddlePointBaseTaylorCoeff}, such that 
\beq\label{eq:rhotutp}
\begin{split}
\bsrho^\bsnu || t_\bsnu^u ||_\cV &\leq \frac{1}{\alpha} \sup_{||v||_\cV = 1} \sum_{j \in \text{supp}\, \bsnu} a_1(\bsrho^{\bsnu-\bse_j} t_{\bsnu-\bse_j}^u,v; \rho_j \kappa_j),\\
\bsrho^\bsnu || t_\bsnu^p ||_\cQ &\leq \frac{\alpha + \gamma}{\alpha \beta} \sup_{||v||_\cV = 1} \sum_{j \in \text{supp}\, \bsnu} a_1(\bsrho^{\bsnu-\bse_j} t_{\bsnu-\bse_j}^u,v; \rho_j \kappa_j),
\end{split}
\eeq
where by \eqref{eq:bounda1} and $|\bsnu|_0 = \#\{j \in \bbN: \nu_j > 0\}< \infty $ for any $\bsnu \in \cF$ we have 
\beq\label{eq:suppnurhoa1}
\begin{split}
\sup_{||v||_\cV = 1} \sum_{j \in \text{supp}\, \bsnu} a_1(\bsrho^{\bsnu-\bse_j} t_{\bsnu-\bse_j}^u,v; \rho_j \kappa_j) \\
\leq C_1 |\bsnu |_0 (||\kappa_{0}||_{\cK} - \epsilon) \max_{j \geq 1} (\bsrho^{\bsnu - \bse_j} ||t_{\bsnu - \bse_j}^u||_\cV) < \infty. 
\end{split}
\eeq
Therefore, by taking the test functions as $(v, q) = (\bsrho^\bsnu t_\bsnu^u, \bsrho^\bsnu t_\bsnu^p)$, we obtain 
\beq\label{eq:sumbound}
\begin{split}
&a(\bsrho^\bsnu t_\bsnu^u,\bsrho^\bsnu t_\bsnu^u;\kappa_0) = - \sum_{j \in \text{supp}\, \bsnu} a_1(\bsrho^{\bsnu-\bse_j} t_{\bsnu-\bse_j}^u,\bsrho^\bsnu t_\bsnu^u; \rho_j \kappa_j) \\
& \leq \frac{1}{2}\sum_{j\in \text{supp}\, \bsnu } a_1(\bsrho^{\bsnu-\bse_j} t_{\bsnu-\bse_j}^u, \bsrho^{\bsnu-\bse_j} t_{\bsnu-\bse_j}^u; \rho_j|\kappa_j|) + a_1(\bsrho^{\bsnu} t_{\bsnu}^u,\bsrho^\bsnu t_\bsnu^u; \rho_j |\kappa_j|),
\end{split}
\eeq
where for the inequality we used the assumption \eqref{eq:bounda1}. Therefore, by \eqref{eq:rhoSum}, we have 
\beq
\sum_{j\in \text{supp}\, \bsnu}a_1(\bsrho^\bsnu t_\bsnu^u,\bsrho^\bsnu t_\bsnu^u; \rho_j |\kappa_j|) \leq a_1(\bsrho^\bsnu t_\bsnu^u,\bsrho^\bsnu t_\bsnu^u; \kappa_0 - \epsilon), 
\eeq
which, together with \eqref{eq:sumbound} leads to 
\beq\label{eq:aa1k0sum}
a(\bsrho^\bsnu t_\bsnu^u,\bsrho^\bsnu t_\bsnu^u; \frac{\kappa_0+\epsilon}{2}) \leq \frac{1}{2}\sum_{j\in \text{supp}\, \bsnu } a_1(\bsrho^{\bsnu-\bse_j} t_{\bsnu-\bse_j}^u, \bsrho^{\bsnu-\bse_j} t_{\bsnu-\bse_j}^u; \rho_j|\kappa_j|). 
\eeq
By Assumption \ref{ass:UEA}, we have 
\beq
a(\bsrho^\bsnu t_\bsnu^u,\bsrho^\bsnu t_\bsnu^u; \theta) \geq \alpha ||\bsrho^\bsnu t_\bsnu^u||_\cV^2 \geq 0,
\eeq
so that by the affine structure \eqref{eq:affineA} there holds
\beq\label{eq:aa1k0}
\begin{split}
a(\bsrho^\bsnu t_\bsnu^u,\bsrho^\bsnu t_\bsnu^u; \frac{\kappa_0+\epsilon}{2}) &= a(\bsrho^\bsnu t_\bsnu^u,\bsrho^\bsnu t_\bsnu^u; \theta) + a_1(\bsrho^\bsnu t_\bsnu^u,\bsrho^\bsnu t_\bsnu^u; \frac{\kappa_0+\epsilon}{2} - \theta)\\
& \geq a_1(\bsrho^\bsnu t_\bsnu^u,\bsrho^\bsnu t_\bsnu^u; \frac{\kappa_0+\epsilon}{2} - \theta).
\end{split}
\eeq
Hence, from \eqref{eq:aa1k0sum} and \eqref{eq:aa1k0} we obtain 
\beq
a_1(\bsrho^\bsnu t_\bsnu^u,\bsrho^\bsnu t_\bsnu^u; \frac{\kappa_0+\epsilon}{2} - \theta) \leq \frac{1}{2}\sum_{j\in \text{supp}\, \bsnu } a_1(\bsrho^{\bsnu-\bse_j} t_{\bsnu-\bse_j}^u, \bsrho^{\bsnu-\bse_j} t_{\bsnu-\bse_j}^u; \rho_j|\kappa_j|).
\eeq
Summing over $|\bsnu| = k$ for any $k \geq 1$ for both sides, we have 
\beq
\begin{split}
&\sum_{|\bsnu|=k}
a_1(\bsrho^\bsnu t_\bsnu^u,\bsrho^\bsnu t_\bsnu^u; \frac{\kappa_0+\epsilon}{2} - \theta)\\
&=\frac{1}{2}\sum_{|\bsnu|=k}\sum_{j\in \text{supp}\, \bsnu } a_1(\bsrho^{\bsnu-\bse_j} t_{\bsnu-\bse_j}^u, \bsrho^{\bsnu-\bse_j} t_{\bsnu-\bse_j}^u; \rho_j|\kappa_j|) \\
&= \frac{1}{2} \sum_{|\bsnu|=k-1} \sum_{j\geq 1} a_1(\bsrho^\bsnu t_\bsnu^u,\bsrho^\bsnu t_\bsnu^u; \rho_j|\kappa_j|)\\
&\leq \sum_{|\bsnu|=k-1} a_1(\bsrho^\bsnu t_\bsnu^u,\bsrho^\bsnu t_\bsnu^u; \frac{\kappa_0 - \epsilon}{2})\\
& \leq \sup_{x \in D} \frac{\kappa_0(x)-\epsilon}{\kappa_0(x) + \epsilon - 2\theta} \sum_{|\bsnu|=k-1} a_1(\bsrho^\bsnu t_\bsnu^u,\bsrho^\bsnu t_\bsnu^u; \frac{\kappa_0+\epsilon}{2} - \theta),
\end{split}
\eeq
where we used Assumption \ref{ass:rhoSum} in the first inequality. By denoting 
\beq
\sigma = \sup_{x \in D} \frac{\kappa_0(x)-\epsilon}{\kappa_0(x) + \epsilon - 2\theta} < 1, \text{ since } \theta < \epsilon,
\eeq
we obtain 
\beq
\sum_{|\bsnu|=k}
a_1(\bsrho^\bsnu t_\bsnu^u,\bsrho^\bsnu t_\bsnu^u; \frac{\kappa_0+\epsilon}{2} - \theta) \leq \sigma^k a_1(t_\bszero^u, t_\bszero^u; \frac{\kappa_0+\epsilon}{2} - \theta).
\eeq
Summing over $k \geq 1$, we have
\beq
\sum_{\bsnu \in \cF} a_1(\bsrho^\bsnu t_\bsnu^u,\bsrho^\bsnu t_\bsnu^u; \frac{\kappa_0+\epsilon}{2} - \theta) \leq \frac{1}{1-\sigma} a_1(t_\bszero^u, t_\bszero^u; \frac{\kappa_0+\epsilon}{2} - \theta) < \infty. 
\eeq
By the coercivity condition \eqref{eq:bounda1} in $ \cV^0$, for any $\bsnu \neq \bszero$, as $t_\bsnu^u \in \cV^0$ we have
\beq\label{eq:l2coercivity}
a_1(\bsrho^\bsnu t_\bsnu^u,\bsrho^\bsnu t_\bsnu^u; \frac{\kappa_0+\epsilon}{2} - \theta) \geq c_1 \inf_{x \in D} \left(\frac{\kappa_0(x)+\epsilon}{2} - \theta\right)(\bsrho^\bsnu ||t_\bsnu^u||_\cV)^2,
\eeq
where $\inf_{x \in D} \kappa_0(x) > \epsilon > \theta$ by Assumption \ref{ass:rhoSum}. Therefore, we obtain
\beq\label{eq:weightedL2}
\sum_{\bsnu \in \cF} (\bsrho^\bsnu ||t_\bsnu^u||_\cV)^2 < \infty.
\eeq
%
%
%
%
%
%
By H\"older's inequality, 
we have 
\beq
\begin{split}
\sum_{\bsnu \in \cF} ||t_\bsnu^u||_\cV^s &= \sum_{\bsnu \in \cF} (\bsrho^\bsnu ||t_\bsnu^u||_\cV)^s \bsrho^{-s\bsnu} \\
&\leq \left( \sum_{\bsnu \in \cF} (\bsrho^\bsnu ||t_\bsnu^u||_\cV)^2 \right)^{s/2} \left(\sum_{\bsnu \in \cF} \bsrho^{-\frac{2s}{2-s} \bsnu}\right)^{(2-s)/2},
\end{split}
\eeq
where the first term is finite by \eqref{eq:weightedL2}. For the second term, with $t = \frac{2s}{2-s}$, i.e., $s = \frac{2t}{2+t}$, we have  
\beq
\sum_{\bsnu \in \cF} \bsrho^{-\frac{2s}{2-s} \bsnu} = \prod_{j\geq 1} \left(\sum_{k = 0}^\infty \rho_j^{-tk}\right) = \prod_{j\geq 1}(1-\rho_j^{-t})^{-1}.
\eeq
As $(\rho_j^{-1})_{j\geq 1} \in \ell^t(\bbN)$, there exists $J \in \bbN$ such that $\rho_j^{-t} < \frac{1}{2}$ for all $j > J$. Note that $g(x) := -\log(1-x) - 2x < 0$ as $g(0) = 0$ and $g'(x) = \frac{1}{1-x} - 2 < 0$ for $0 < x < \frac{1}{2}$, which implies $(1-\rho_j^{-t})^{-1} < \exp(2\rho_j^{-t})$ for $j > J$, so that
\beq
 \prod_{j\geq 1} (1-\rho_j^{-t})^{-1} <   \exp\left(2 \sum_{j > J}\rho_j^{-t}\right) \prod_{j\leq J } (1-\rho_j^{-t})^{-1}, 
\eeq 
which is finite as $(\rho_j^{-1})_{j\geq 1} \in \ell^t(\bbN)$. Therefore, $(||t_\bsnu^u||_\cV)_{\bsnu \in \cF} \in \ell^s(\cF)$.

By \eqref{eq:weightedL2}, there exists a constant $C_2 > 0$ such that 
\beq\label{eq:tnuu}
\sup_{\bsnu \in \cF}  ||t_\bsnu^u||_\cV \leq C_2 \bsrho^{-\bsnu}.
\eeq
Therefore, by \eqref{eq:rhotutp} and \eqref{eq:suppnurhoa1}, we have 
\beq\label{eq:tnup}
||t_\bsnu^p||_\cQ \leq C_3 \bsrho^{-\bsnu} |\bsnu |_0  \leq  C_3 \bsrho^{-\bsnu} \prod_{j\geq 1}(1+\nu_j)
\eeq
where
\beq
C_3 = C_1 C_2 \frac{\alpha + \gamma}{\alpha \beta} (||\kappa_0 ||_{\cK} - \epsilon) < \infty,
\eeq
and we used the fact $|\bsnu|_0 \leq \prod_{j\geq 1}(1+\nu_j)$ for any $\bsnu \in \cF$ in the second inequality.
Hence, we have
\beq\label{eq:tnuplt}
\sum_{\bsnu \in \cF} ||t_\bsnu^p||^t_\cQ \leq (C_3)^t \sum_{\bsnu \in \cF} \prod_{j\geq 1} \rho_j^{-t\nu_j} (1+\nu_j)^t = (C_3)^t \prod_{j\geq 1} \sum_{k = 0}^\infty \rho_j^{-t k} (1+k)^t,
\eeq
where for each $j \geq 0$ we have
\beq
\sum_{k = 0}^\infty \rho_j^{-t k} (1+k)^t  = 1 + \rho_j^{-t}\sum_{k = 0} \rho_j^{-tk} (2+k)^t. 
\eeq
As $(\rho_j^{-1})_{j\geq 1} \in \ell^t(\bbN)$, there exists $J > 0$ such that $\rho_j^{-1} < \frac{1}{4}$ for any $j > J$. Moreover, for any $t > 0$, there exist $c_1 > 0$ and $1 < c_2 < 2$ such that $(2+k)^t \leq c_1 c_2^k$ for $k \geq 0$, so that 
\beq\label{eq:c2}
\sum_{k = 0}^\infty \rho_j^{-tk} (2+k)^t \leq c_1 \sum_{k = 0}^\infty (\rho_j^{-1} c_2)^k = c_1 (1- \rho_j^{-1}c_2)^{-1} \leq 2 c_1.
\eeq
As $\rho_j > 1$, there exists $C_j < \infty$ for each $j \geq 1$ such that 
\beq
\sum_{k = 0}^\infty \rho_j^{-t k} (1+k)^t \leq C_j.
\eeq
Therefore, we have
\beq\label{eq:rhot}
\prod_{j\geq 1} \sum_{k = 0}^\infty \rho_j^{-t k} (1+k)^t \leq \prod_{j \leq J} C_j \prod_{j > J} (1+ 2c_1 \rho_j^{-t}) \leq \exp\left(2c_1 \sum_{j > J} \rho_j^{-t}\right) \prod_{j \leq J} C_j 
\eeq
which is finite when $(\rho_j^{-1})_{j\geq 1} \in \ell^t(\bbN)$. Note that in the second inequality, we used $1 + x \leq e^x$ for $x \geq 0$. Hence $(||t_\bsnu^p||_\cQ)_{\bsnu \in \cF} \in \ell^t(\cF)$ from \eqref{eq:tnuplt}.
\end{proof}

\begin{remark}
We remark that the weighted $\ell^2$-summability for $(||t_\bsnu^u||_\cV)_{\bsnu \in \cF}$ in Lemma \ref{lemma:wLsL2} is a result of the coercivity property \eqref{eq:l2coercivity} (where the $\ell^2$-norm shows up) of the bilinear form $a_1(\cdot, \cdot;\kappa): \cV \times \cV \to \bbR$. 
However, the weighted $\ell^2$-summability cannot be shown for  $(||t_\bsnu^p||_\cQ)_{\bsnu \in \cF}$, where $t_\bsnu^p$ only appears in the bilinear form $b(\cdot, \cdot): \cV \times \cQ \to \bbR$ that holds the inf-sup condition. Instead, by this condition, we can bound the Taylor coefficient $t_\bsnu^p$ as in \eqref{eq:tnup} by \eqref{eq:suppnurhoa1}. 

\end{remark}

\subsection{Dimension-independent convergence}

As a consequence of the summability obtained in the Sec.\ \ref{sec:analyticRegularity} and Sec.\ \ref{sec:weightedL2}, we obtain the following convergence results.

\begin{theorem}
\label{thm:bestNterm}
Under Assumption \ref{ass:saddlePointBase} and \ref{ass:UEA}, there exist two sequences of index sets $(\Lambda^u_N)_{N \geq 1}$ and $(\Lambda^p_N)_{N \geq 1}$ with indices $\bsnu \in \cF$ corresponding to the $N$ largest Taylor coefficients $||t_\bsnu^u||_\cV$ and $||t_\bsnu^p||_\cQ$, respectively, such that 
\beq\label{eq:ConvRate3}
\begin{split}
\sup_{\bsy \in U}||u(\bsy) - T_{\Lambda^u_N} u(\bsy)||_\cV &\leq ||(||t_\bsnu^u||_\cV)_{\bsnu \in \cF}||_{\ell^s(\cF)} N^{-r(s)}, \\
\sup_{\bsy \in U}||p(\bsy) - T_{\Lambda^p_N} p(\bsy)||_\cQ &\leq ||(||t_\bsnu^p||_\cQ)_{\bsnu \in \cF}||_{\ell^s(\cF)} N^{-r(s)}, 
\end{split}
\eeq
 under Assumption \ref{ass:lpSum}, and 
\beq\label{eq:ConvRate4}
\begin{split}
\sup_{\bsy \in U}||u(\bsy) - T_{\Lambda^u_N} u(\bsy)||_\cV &\leq ||(||t_\bsnu^u||_\cV)_{\bsnu \in \cF}||_{\ell^s(\cF)} N^{-r(s)}, \\
\sup_{\bsy \in U}||p(\bsy) - T_{\Lambda^p_N} p(\bsy)||_\cQ &\leq ||(||t_\bsnu^p||_\cQ)_{\bsnu \in \cF}||_{\ell^t(\cF)} N^{-r(t)}, 
\end{split}
\eeq
under Assumption \ref{ass:rhoSum}, where the dimension-independent convergence rate $r$ is given by
\beq\label{eq:rs}
r(s) = \frac{1}{s} - 1, \quad s < 1.
\eeq
\end{theorem}

\begin{proof}
At first, by Lemma \ref{lemma:ParaTaylorSum} and Lemma \ref{lemma:wLsL2} for Assumption \ref{ass:lpSum} and Assumption \ref{ass:rhoSum}, respectively, for any $s < 1$, we have 
\beq
\sup_{\bsy \in U}\left\|\sum_{\bsnu \in \cF} \bsy^\bsnu t_\bsnu^u \right\|_\cV \leq \sup_{\bsy \in U}\sum_{\bsnu \in \cF}|\bsy^\bsnu| \, ||t_\bsnu^u ||_\cV \leq \sum_{\bsnu \in \cF} ||t_\bsnu^u ||_\cV \leq \sum_{\bsnu \in \cF} ||t_\bsnu^u ||_\cV^s < \infty,
\eeq
which implies that the Taylor power series $T_{\cF} u$ defined in \eqref{eq:TaylorSeries} is uniformly convergent. 
Secondly, for any $\bsy \in U$ and $\varepsilon > 0$, by Lemma \ref{lemma:pertubation}, there exists $J_1 > 0$ such that for any $J \geq J_1$
\beq
B_1 :=||u(\bsy) - u(\bsy_J^0)||_\cV \leq \frac{1}{\alpha} C_1C_u ||\kappa(\bsy) - \kappa(\bsy_J^0)||_{\cK}  < \frac{\varepsilon}{2},
\eeq
under Assumption \ref{ass:lpSum} or \ref{ass:rhoSum}, where $\bsy_J^0$ is defined in the same way as in \eqref{eq:zJ0}. Moreover,  for any $J \geq J_1$, by the analytic regularity of $u(\bsy_J^0)$ in the complex domain $\cU_\bsrho$ as indicated in Lemma \ref{lemma:complexDeri}, there exists $K > 0$ such that for any $\Lambda = \{\bsnu \in \cF: \nu_j > K, \text{ for } j \leq J \text{ and } \nu_j = 0 \text{ for } j > J\}$ there holds
\beq
B_2 :=||u(\bsy_J^0) - T_{\Lambda}u(\bsy_J^0)||_\cV < \frac{\varepsilon}{2}.
\eeq
By definition of $\Lambda$ we have $T_{\Lambda}u(\bsy_J^0) = T_{\Lambda} u(\bsy)$. Hence, we have 
\beq
||u(\bsy) - T_{\Lambda}u(\bsy)||_\cV \leq B_1 + B_2  <  \varepsilon, 
\eeq
which implies that the Taylor power series $T_\cF u(\bsy) $ converges to  $u(\bsy)$ for every $\bsy \in U$. Consequently,
\beq
\sup_{\bsy \in U}||u(\bsy) - T_{\Lambda^u_N} u(\bsy)||_\cV = \sup_{\bsy \in U} \left\| \sum_{\bsnu \not \in  \Lambda^u_N} \bsy^\bsnu t_\bsnu^u \right\|_\cV \leq \sum_{\bsnu \not \in  \Lambda^u_N} ||t_\bsnu^u||_\cV,
\eeq
which concludes for the error of the Taylor approximation of $u$ by using Stechkin's Lemma \cite[Lemma 5.5]{CohenDeVoreSchwab2010}, i.e., for a non-increasing arrangement of $( ||t_\bsnu^u||_\cV)_{\bsnu \in \cF}$, there holds
\beq
\sum_{\bsnu \not \in  \Lambda^u_N} ||t_\bsnu^u||_\cV \leq \left(\sum_{\bsnu \in \cF} ||t_\bsnu^u||_\cV^s \right)^{1/s} N^{-r(s)},
\eeq
with $r(s)$ defined in \eqref{eq:rs}. The same result holds for the error of the Taylor approximation of $p$ by using the same argument. 
\end{proof}

\begin{remark}
We remark that the convergence results \eqref{eq:ConvRate3} and \eqref{eq:ConvRate4} are obtained under different assumptions, and cannot be implied by one another. In fact, it is clear that \eqref{eq:ConvRate4} cannot be implied by \eqref{eq:ConvRate3} as explained in Remark \ref{remark:KLtoLocal}. On the other hand, \eqref{eq:ConvRate3} cannot be implied by \eqref{eq:ConvRate4} as shown in the following simple example: let $\kappa_0 = 1$ and $\kappa_j = j^{-2}$ for $j \geq 1$, then by \eqref{eq:ConvRate3} we have the convergence rate $N^{-r}$ for any $r < 1$ arbitrarily close to $1$. However, by \eqref{eq:ConvRate4}, for which there exists $(\rho_j^{-1})_{j\geq 1} \in \ell^t(\bbN)$ with $t > 1$ satisfying \eqref{eq:rhoSum}, we can only obtain a convergence rate of $N^{-r}$ for $r  = \frac{1}{s} - 1 = \frac{1}{t} - \frac{1}{2}< \frac{1}{2}$ for $\sup_{\bsy \in U}||u(\bsy) - T_{\Lambda^u_N} u(\bsy)||_\cV$, and $r = \frac{1}{t} - 1<0$, i.e., non-convergent, for $\sup_{\bsy \in U}||p(\bsy) - T_{\Lambda^p_N} p(\bsy)||_\cQ$. 

\end{remark}

Theorem \ref{thm:bestNterm} states the existence of such index sets $\Lambda^u_N \subset \cF$ and $\Lambda^p_N \subset \cF$ that lead to the dimension-independent convergence rates. However, there is no particular structure of these index sets. To guide more practical algorithm development, we consider a particular structure of these index sets, namely, downward closed set $\Lambda \subset \cF$, also known as admissible set or monotone set \cite{ChkifaCohenDeVoreEtAl2013, GerstnerGriebel2003, ChkifaCohenSchwab2014}, which satisfies 
\beq
\text{ if } \bsnu \in \Lambda \text{ then } \bsmu \in \Lambda, \; \forall \bsmu \preceq \bsnu,
\eeq
where we recall that $\bsmu \preceq \bsnu$ means $\mu_j \leq \nu_j$, for all $j \geq 1$. 

We say that a sequence $(\theta_\bsnu)_{\bsnu \in \cF}$ is monotonically decreasing 
\beq
\text{ if }  \bsmu \preceq \bsnu \text{ then } \theta_\bsnu \leq \theta_{\bsmu}.
\eeq

\begin{lemma}
\label{lemma:monotoneLemma}
Let $(\theta_\bsnu)_{\bsnu \in \cF}$ be a monotonically decreasing sequence of positive real numbers in $\ell^s(\cF)$ with $s < 1$, then there exists a sequence of downward closed and nested index sets $(\Lambda_N)_{N\geq 1} \subset \cF$ such that 
\beq\label{eq:thetaResidual}
\sum_{\bsnu \not \in \Lambda_N} \theta_\bsnu \leq ||(\theta_\bsnu)_{\bsnu \in \cF}||_{\ell^s(\cF)} N^{-r(s)}, \quad r(s) = \frac{1}{s} - 1.
\eeq
\end{lemma}

\begin{proof}
By Stechkin's Lemma as in the proof of Theorem \ref{thm:bestNterm}, there exists a sequence of index sets $(\Lambda_N)_{N\geq 1} \subset \cF$ such that \eqref{eq:thetaResidual} holds. It is left to show that $(\Lambda_N)_{N\geq 1}$ can be taken as downward closed and nested. This is achieved by an induction argument. First, for $N = 1$, we take $\Lambda_1 = \{\bsnu(1)\}$ with $\bsnu(1) = \bszero$, then \eqref{eq:thetaResidual} holds. Suppose \eqref{eq:thetaResidual} holds for some $N > 1$ with downward closed and nested index set $\Lambda_N$, then we look for the next index $\bsnu(N+1) \in \cF$ such that $\Lambda_{N+1} := \Lambda_N \cup \{\bsnu(N+1)\}$ is downward closed and \eqref{eq:thetaResidual} holds in $\Lambda_{N+1}$. Let $\cN(\Lambda_N)$ denote the admissible forward neighbor set defined as 
\beq
\cN(\Lambda_N) = \{\bsnu \in \cF \setminus \Lambda_N: \bsnu - \bse_j \in \Lambda_N \text{ for every } j \in \bbN \text{ such that } \nu_j \neq 0\},
\eeq
where we recall the Kronecker sequence $\bse_j = (\delta_{ij})_{i \geq 1}$.
Then we take 
\beq
\bsnu(N+1) = \argmax_{\bsmu \in \cN(\Lambda_N)} \theta_\bsmu.
\eeq
By definition of the admissible forward neighbor set $\cN(\Lambda_N)$, we have $\Lambda_{N+1} := \Lambda_N \cup \{\bsnu\}$ is downward closed for any $\bsnu \in \cN(\Lambda_N)$. Moreover, the sequence $\theta_{\bsnu(N)}$ is monotonically decreasing as  $\bsnu(N) \preceq \bsnu(N+1)$ for every $N \geq 1$, which concludes. 
\end{proof}

Let $(\theta_\bsnu)_{\bsnu\in \cF}$ be a real sequence. Then the sequence $(\theta_\bsnu^*)_{\bsnu \in \cF}$ with 
\beq\label{eq:monotoneEnvelope}
\theta_\bsnu^* := \max_{\bsnu \preceq \bsmu} \theta_\bsmu, \quad \forall \bsnu \in \cF,
\eeq 
is monotonically decreasing. If the sequence $(\theta_\bsnu^*)_{\bsnu \in \cF}$ is $\ell^s(\cF)$-summable, then we denote a $\ell^s_m(\cF)$-norm for $(\theta_\bsnu)_{\bsnu \in \cF}$ as 
\beq
||(\theta_\bsnu)_{\bsnu \in \cF}||_{\ell^s_m(\cF)} = ||(\theta^*_\bsnu)_{\bsnu \in \cF}||_{\ell^s(\cF)}.
\eeq

\begin{theorem}
\label{thm:bestNtermDC}
Under Assumption \ref{ass:saddlePointBase} and \ref{ass:UEA}, there exist two sequences of downward closed and nested index sets $(\Lambda^u_N)_{N \geq 1}$ and $(\Lambda^p_N)_{N \geq 1}$ with indices $\bsnu \in \cF$ corresponding to the $N$ largest Taylor coefficients $||t_\bsnu^u||_\cV$ and $||t_\bsnu^p||_\cQ$, respectively, such that 
\beq
\begin{split}
\sup_{\bsy \in U}||u(\bsy) - T_{\Lambda^u_N} u(\bsy)||_\cV &\leq ||(||t_\bsnu^u||_\cV)_{\bsnu \in \cF}||_{\ell^s_m(\cF)} N^{-r(s)}, \\
\sup_{\bsy \in U}||p(\bsy) - T_{\Lambda^p_N} p(\bsy)||_\cQ &\leq ||(||t_\bsnu^p||_\cQ)_{\bsnu \in \cF}||_{\ell^s_m(\cF)} N^{-r(s)}, 
\end{split}
\eeq
 under Assumption \ref{ass:lpSum}, and 
\beq
\begin{split}
\sup_{\bsy \in U}||u(\bsy) - T_{\Lambda^u_N} u(\bsy)||_\cV &\leq ||(||t_\bsnu^u||_\cV)_{\bsnu \in \cF}||_{\ell^t_m(\cF)} N^{-r(t)}, \\
\sup_{\bsy \in U}||p(\bsy) - T_{\Lambda^p_N} p(\bsy)||_\cQ &\leq ||(||t_\bsnu^p||_\cQ)_{\bsnu \in \cF}||_{\ell^t_m(\cF)} N^{-r(t)}, 
\end{split}
\eeq
under Assumption \ref{ass:rhoSum}, where the dimension-independent convergence rate $r$ is given by
\beq
r(s) = \frac{1}{s} - 1, \quad s < 1.
\eeq
\end{theorem}

\begin{proof}
By Theorem \ref{thm:bestNterm} and Lemma \ref{lemma:monotoneLemma}, we only need to show that $(||t_\bsnu^u||_\cV^*)_{\bsnu \in \cF}$ and $(||t_\bsnu^p||_\cQ^*)_{\bsnu \in \cF}$, the associated monotone envelopes defined in \eqref{eq:monotoneEnvelope} for $(||t_\bsnu^u||_\cV)_{\bsnu \in \cF}$ and $(||t_\bsnu^p||_\cQ)_{\bsnu \in \cF}$, respectively, are $\ell^s(\cF)$-summable under Assumption \ref{ass:lpSum}, and $\ell^t(\cF)$-summable under Assumption \ref{ass:rhoSum}. Under Assumption \ref{ass:lpSum}, by Lemma \ref{lemma:TaylorCoeffBound} we have 
\beq
||t_\bsnu^u||_\cV^* \leq C_u \bsrho^{-\bsnu} \text{ and } ||t_\bsnu^p||_\cQ^* \leq C_p \bsrho^{-\bsnu}, \quad \forall \bsnu \in \cF,
\eeq
since $(\bsrho^{-\bsnu})_{\bsnu \in \cF}$ is monotonically decreasing by \eqref{eq:rhor}. Moreover, as shown in Lemma \ref{lemma:ParaTaylorSum}, $(\bsrho^{-\bsnu})_{\bsnu \in \cF}$ is $\ell^s(\cF)$-summable, which concludes. Under Assumption \ref{ass:rhoSum}, we have by \eqref{eq:tnuu} and \eqref{eq:tnup} that 
\beq\label{eq:tupstar}
||t_\bsnu^u||_\cV^* \leq C_2 \bsrho^{-\bsnu} \text{ and } ||t_\bsnu^p||_\cQ^* \leq C_3 \theta_\bsnu^*, \quad \forall \bsnu \in \cF,
\eeq
since both $(\bsrho^{-\bsnu})_{\bsnu \in \cF}$ and $(\theta_\bsnu^*)_{\bsnu \in \cF}$ are monotonically decreasing, where we denote 
\beq
\theta_\bsnu = \bsrho^{-\bsnu} \prod_{j\geq 1} (1+\nu_j), \quad \bsnu \in \cF.
\eeq
The $\ell^t(\cF)$-summability of $(\bsrho^{-\bsnu})_{\bsnu \in \cF}$ can be shown as in \eqref{eq:tnuplt}. For the $\ell^t(\cF)$-summability of $(\theta_\bsnu^*)_{\bsnu \in \cF}$, we proceed as follows. 
As $(\rho_j^{-1})_{j\geq 1} \in \ell^t(\bbN)$, there exists a $J$ such that $\rho_j^{-1} < 1/4$ for all $j > J$, which implies  
\beq\label{eq:monobound1}
\frac{\theta_{\bsnu + \bse_j} }{\theta_{\bsnu}} = \frac{(1+\nu_j + 1)}{(1+\nu_j)\rho_j } < 1 
, \quad \forall j > J.
\eeq
Moreover, as $\rho_j > 1$ there exists $K \in \bbN$ such that $(1+k+1)/(1+k) < \rho_j$ for all $j \leq J$ when $k > K$, so that 
\beq\label{eq:monobound2}
\frac{\theta_{\bsnu + \bse_j} }{\theta_{\bsnu}} = \frac{(1+\nu_j + 1)}{(1+\nu_j)\rho_j } < 1, \quad \forall j \leq J \text{ and }  \nu_j > K.
\eeq
By defining a sequence of functions $(\theta^{(J, K)}_j)_{j\geq 1}$ as
\beq
\theta^{(J, K)}_{j} (k) = 
\left\{
\begin{array}{ll}
\max_{k \leq K} \rho_j^{-k} (1+k) & j \leq J \text{ and } k \leq K,\\
\rho_j^{-k}(1+k) & j > J \text{ or } k > K,
\end{array}
\right.
\eeq
and defining a new sequence $(\Theta_\bsnu)_{\bsnu \in \cF}$ as 
\beq
\Theta_\bsnu := \prod_{j \geq 1} \theta^{(J, K)}_j(\nu_j), \quad \forall \bsnu \in \cF, 
\eeq
we have that $(\Theta_\bsnu)_{\bsnu \in \cF}$ is monotonically decreasing by \eqref{eq:monobound1} and \eqref{eq:monobound2}. Moreover, 
the monotone envelope of $(\theta_\bsnu)_{\bsnu \in \cF}$ satisfies $\theta^*_\bsnu \leq \Theta_\bsnu $ for all $\bsnu \in \cF$. Therefore, we only need to show $(\Theta_\bsnu)_{\bsnu \in \cF} \in \ell^t(\cF)$. By definition we have 
\beq\label{eq:Thetalt}
\sum_{\bsnu \in \cF} \Theta_\bsnu^t = \sum_{\bsnu \in \cF} \prod_{j\geq 1} (\theta_j^{(J, K)}(\nu_j))^t = \prod_{1\leq j \leq J} \sum_{k = 0}^\infty (\theta_j^{(J, K)}(k))^t \prod_{j > J} \sum_{k = 0}^\infty (\theta_j^{(J, K)}(k))^t.
\eeq
Since $\rho_j > 1$, there exist a constant $C^{(K, J)}_j < \infty $ for each $j \geq 1$ such that 
\beq
\sum_{k = 0}^\infty (\theta_j^{(J, K)}(k))^t = K \max_{k\leq K} \rho_j^{-tk} (1+k)^t + \sum_{k = K+1}^\infty \rho_j^{-tk}(1+k)^t < C_j^{(K, J)}, 
\eeq
Therefore, the first term of \eqref{eq:Thetalt} can be bounded as 
\beq
\prod_{1\leq j \leq J} \sum_{k = 0}^\infty (\theta_j^{(J, K)}(k))^t \leq \prod_{1\leq j \leq J} C_j^{(K, J)} < \infty.  
\eeq
The second term of \eqref{eq:Thetalt} can be bounded as in \eqref{eq:rhot}, i.e., 
\beq
\prod_{j > J} \sum_{k = 0}^\infty (\theta_j^{(J, K)}(k))^t = \prod_{j > J} \sum_{k = 0}^\infty \rho_j^{-tk} (1+k)^t  \leq \exp\left(2c_1 \sum_{j > J} \rho_{j}^{-1}\right),
\eeq
which is finite when $(\rho_j^{-1})_{j\geq 1} \in \ell^t(\bbN)$. 
Hence, $(\Theta_\bsnu)_{\bsnu \in \cF} \in \ell^t(\cF)$, which concludes. 
\end{proof}

\begin{remark}
Note that the same convergence rate is obtained in Theorem \ref{thm:bestNtermDC} for downward closed and nested index sets as in Theorem \ref{thm:bestNterm} for more general index sets under Assumption \ref{ass:lpSum}. While under Assumption \ref{ass:rhoSum}, the convergence rates for the Taylor approximation of $u$ becomes different. Specifically, the convergence rate from $N^{-r(s)}$ is deteriorated to $N^{-r(t)}$ with $r(s) > r(t)$, as $s = \frac{2t}{2+t} < t $, for downward closed and nested index sets. This deterioration is due to the bound \eqref{eq:tnuu}, which may be crude and the convergence rate may not be optimal. 
\end{remark}

\section{Conclusions}

We studied sparse polynomial approximations for parametric saddle
point problems, which covered such problems as Stokes, mixed
formulation of the Poisson, and time-harmonic Maxwell problems. We
considered the setting of a random input parameter parametrized by a
countably infinite number of independent parameters as the
coefficients of an affine expansion on a series of basis
functions. Both globally and locally supported basis functions were
considered, which led to different assumptions on the sparsity of the
parametrization. Based on the two different sparsity assumptions, we
established the $\ell^s$-summability of the coefficients of the Taylor
expansion of the parametric solutions by different
approaches---analytic regularity and weighted $\ell^2$-summability,
respectively. By the $\ell^s$-summability we proved the
dimension-independent algebraic convergence rates of the sparse
polynomial approximations, thus breaking the curse of dimensionality
for high or infinite dimensional parametric saddle point problems.
Moreover, we considered sparse polynomial approximations of the
parametric solutions on downward closed and nested multi-index sets
and demonstrated the dimension-independent convergence rates too.

The convergence results obtained in this work establish a theoretical
foundation for the development and application of computational
algorithms such as adaptive, least-squares, and compressive sensing
constructions of sparse polynomial approximations, and the sparse
polynomial based interpolation and integration addressed in
\cite{ChenGhattas2018b}. Moreover, they can serve as a guideline for
error estimates of model reduction techniques such as reduced basis
methods constructed by greedy algorithms \cite{ChenSchwab2015}, which
will be addressed for high-dimensional parametric saddle point
problems elsewhere. Note that we only considered uniformly distributed
parameters in this work. We are interested in studying more general
distributions such as Gaussian or log-normal random fields 
for saddle point problems, motivated by their
recent analysis for elliptic PDEs \cite{BachmayrCohenDeVoreEtAl2017,
  ChenVillaGhattas2017, ErnstSprungkTamellini2018}. Finally, we
mention a particular type of parametric saddle point
problem---optimality systems arising from stochastic PDE-constrained
optimal control \cite{KunothSchwab2013,
ChenQuarteroni2014, ChenQuarteroniRozza2016}. Extending the convergence analysis from
simple cost functionals involving only the mean of the objective
function to more general risk measures is of practical interest.

\end{document}


\maketitle
\slugger{sisc}{xxxx}{xx}{x}{x--x}

\pagestyle{myheadings}
\thispagestyle{plain}
\markboth{Supplementary materials for sparse polynomial approximations for affine parametric saddle point problems}{P. Chen}

Let $D\subset \bbR^d$ ($d =  2, 3$) be an open and bounded physical domain with Lipschitz continuous boundary $\partial D = \Gamma$, which can be split to Dirichlet boundary $\Gamma_0$ and Neumann boundary $\Gamma_1$ such that $\Gamma = \Gamma_0 \cup \Gamma_1$ and $\Gamma_0 \cap \Gamma_1 = \emptyset$. Let $L^\infty(D)$ denote a space of essentially bounded measurable functions, i.e.,  
\beq
L^\infty(D) = \left\{v: \esssup_{x \in D} |v(x)| = ||v||_{L^\infty(D)}< \infty \right\}.
\eeq
Let $L^2(D)$ denote a space of square integrable functions on $D$, i.e., 
\beq
L^2(D) = \left\{v: \int_D |v|^2 dx = ||v||^2_{L^2(D)} < \infty\right\}.
\eeq 
Let $\nabla $, $\nabla \cdot $, $\nabla \times $ denote the gradient, divergence, and curl operators. We use the definition of the following Hilbert spaces by convention \cite{BoffiBrezziFortin2013}
\beq
\begin{aligned}
H^1(D) & := \left\{v \in L^2(D): |\nabla v| \in L^2(D)\right\},\\
H(\text{div};D) & := \left\{\bsv \in (L^2(D))^d: \nabla \cdot  \bsv \in L^2(D)\right\},\\
H(\text{curl}; D) & := \left\{\bsv \in (L^2(D))^d: \nabla \times \bsv \in (L^2(D))^d\right\},
\end{aligned}
\eeq
with corresponding norms 
\beq
\begin{aligned}
||v||_{H^1(D)}^2 & := ||v||_{L^2(D)}^2 + ||\nabla v||_{L^2(D)}^2,\\
||\bsv||_{H(\text{div};D)}^2 & := ||\bsv||_{(L^2(D))^d}^2 + ||\nabla \cdot \bsv ||_{L^2(D)}^2,\\
||\bsv||_{H(\text{curl}; D)}^2 & := ||\bsv||_{(L^2(D))^d}^2 + ||\nabla \times \bsv ||_{L^2(D)}^2.
\end{aligned}
\eeq
Moreover, for functions with vanishing values on $\Gamma_0$, we define
\beq
\begin{aligned}
H^1_0(D) & := \left\{v \in H^1(D): v = 0 \text{ on } \Gamma_0\right\},\\
H_0(\text{div};D) & := \left\{v \in H(\text{div};D): \bsv \cdot  \bsn = 0 \text{ on } \Gamma_0 \right\},\\
H_0(\text{curl}; D) & := \left\{\bsv \in H(\text{curl}; D): \bsv \times \bsn = 0 \text{ on } \Gamma_0 \right\}. 
\end{aligned}
\eeq
where $\bsn$ is the unit normal vector along the boundary.
In what follows, we present several classical problems in (mixed) variational formulations. These formulations are preferred due to several reasons \cite{BoffiBrezziFortin2013}: the presence of a physical constraint, physical importance of the variables appearing in the formulations, better accommodation of finite dimensional approximation and/or available data. For simplicity of the presentation, we assume homogeneous Dirichlet and/or Neumann boundary conditions for all the examples.

\section{Stokes flow}
We consider a flow of a viscous incompressible fluid with low velocity in a domain $D$, which can be described by Stokes equations in the variational form as: given parameter $\kappa \in L^\infty(D)$, data $\bsf \in (L^2(D))^d$, find $(\bsu, p) \in (H^1_0(D))^d \times L^2(D)$ such that 
\beq\label{eq:Stokes}
\left\{
\begin{aligned}
\int_D 2\kappa \bsvarepsilon(\bsu):\bsvarepsilon(\bsv) dx - \int_D (\nabla \cdot \bsv) \, p dx &= \int_D \bsf \cdot \bsv dx  \quad \forall \bsv \in (H^1_0(D))^d,\\
\int_D (\nabla \cdot \bsu) \, q dx &= 0 \quad \forall q \in L^2(D),
\end{aligned}
\right.
\eeq
where $\bsu$ is the velocity, $p$ is the pressure, $\kappa > 0$ is the shear viscosity, $\bsf \in \bbR^d$ is the body force, and $\bsvarepsilon(\bsu) \in \bbR^{d\times d}$ is the strain rate tensor defined as
\beq
\bsvarepsilon(\bsu) := \frac{1}{2}\left(\nabla \bsu + \nabla \bsu^T\right).
\eeq

Problem \eqref{eq:Stokes} can be identified in the abstract saddle point formulation \cite[Eq. 1]{ChenGhattas2018} in the spaces $\cK = L^\infty(D)$, $\cV = (H^1_0(D))^d$ and $\cQ = L^2(D)$ with the bilinear forms 
\beq
\begin{split}
a(\bsw, \bsv; \kappa ) &:= \int_D 2\kappa \bsvarepsilon(\bsw):\bsvarepsilon(\bsv) dx, \quad \forall \bsw, \bsv \in \cV, \\
b(\bsv, q) &:=  - \int_D (\nabla \cdot \bsv) \, q dx, \quad \forall \bsv \in \cV, \forall q \in \cQ.
\end{split}
\eeq
Then \cite[Assumption 1]{ChenGhattas2018} is satisfied with the constants
\beq\label{eq:StokesConst}
\gamma = 2\gamma_2 \esssup_{x\in D} \kappa(x), \quad \delta = 1, \quad \alpha = 2\gamma_1 \essinf_{x\in D} \kappa(x), \text{ and } \beta = \frac{1}{\sqrt{1+C_p}},
\eeq
where the constants $\gamma_1$, $\gamma_2$ are determined by the Korn's inequality \cite{Horgan1995}, i.e.,
\beq
\gamma_1 ||\bsv ||^2_{\cV} \leq \int_D  \bsvarepsilon(\bsv):\bsvarepsilon(\bsv) dx \leq \gamma_2 ||\bsv ||^2_{\cV}, \quad \forall \bsv \in \cV,
\eeq
and $C_p$ is determinted by the Poincar\'e's inequality \cite{Quarteroni2013}, i.e., 
\beq
\int_D |\bsv|^2 dx \leq C_p\int_D |\nabla \cdot \bsv |^2 dx, \quad \forall \bsv \in \cV.
\eeq
Thus the inf-sup constant $\beta $ is obtained as: for any $q \in \cQ$, by taking $\nabla \cdot \bsv = q$,
\beq
\sup_{\bsv \in \cV} \frac{b(\bsv, q)}{||\bsv||_\cV ||q||_\cQ} \geq \frac{ ||q||_\cQ^2}{||\bsv||_\cV ||q||_\cQ} =   \frac{||\nabla \cdot \bsv||_{L^2(D)}}{||\bsv||_{H^1(D)}} \geq \frac{1}{\sqrt{1+C_p}} = : \beta,
\eeq
Therefore, \cite[Theorem 1]{ChenGhattas2018} holds for the Stokes problem \eqref{eq:Stokes} with these constants.


\section{Diffusion}
Diffusion equations are widely used in modelling various physical phenomena. In many applications it is the flux rather than the state that is of interest. For instance in thermo-diffusion problems heat flux may be more important than the temperature field. For such consideration,
we present the diffusion problem in the variational formulation: given parameter $\kappa \in L^\infty(D)$, and data $f\in L^2(D)$, find $(\bsu, p) \in H_0(\text{div};D) \times L^2(D)$ such that 
\beq\label{eq:diffusion}
\left\{
\begin{aligned}
\int_D \kappa \bsu\cdot \bsv dx + \int_D (\nabla \cdot \bsv) p dx &= 0  \quad \forall \bsv \in H_0(\text{div};D),\\
\int_D (\nabla \cdot \bsu) q dx &= -\int_D f q dx \quad \forall q \in L^2(D),
\end{aligned}
\right.
\eeq
where $p$ is the state, e.g., temperature field, the auxiliary variable $\bsu = \kappa^{-1} \nabla p$ represents the flux, $\kappa>0$ is the (inverse) diffusion coefficient, $f$ is a source term. 
With the bilinear forms 
\beq
\begin{split}
a(\bsw, \bsv; \kappa ) &:= \int_D \kappa \bsu\cdot \bsv dx, \quad \forall \bsw, \bsv \in \cV, \\
b(\bsv, q) &:=  \int_D (\nabla \cdot \bsv) q dx, \quad \forall \bsv \in \cV, \forall q \in \cQ,
\end{split}
\eeq
in the Hilbert spaces $\cV = H_0(\text{div};D)$ and $\cQ =  L^2(D)$, we can identify the diffusion problem \eqref{eq:diffusion} in the abstract saddle point formulation \cite[Eq. 1]{ChenGhattas2018} with $\cK = L^\infty(D)$. \cite[Assumption 1]{ChenGhattas2018} is satisfied with the following constants 
\beq\label{eq:DiffusionConst}
\gamma = \esssup_{x\in D} \kappa(x), \quad \delta = 1, \quad \alpha = \essinf_{x\in D} \kappa(x), \text{ and } \beta = \frac{1}{\sqrt{1+C_p}},
\eeq
where $\beta$ is obtained the same as in the Stokes problem.
Note that the bilinear form $a(\cdot, \cdot; \kappa)$ is coercive in $\cV^0$, in which $\nabla \cdot \bsv$ vanishes, even it is not coercive in $\cV$.

\section{Time harmonic Maxwell system}
The foundation of classical electromagnetism, optics, and electric circuits can be described by Maxwell equations. The time harmonic Maxwell system is considered when the propagation of electromagnetic waves at a given frequency is studied or when the Fourier transform in time is used. In the mixed variational formulation, the Maxwell system can be stated as: given parameter $\kappa \in L^\infty(D)$, and data $\bsf \in (L^2(D))^d$, find $(\bsu, p) \in H_0(\text{curl};D)\times H^1_0(D)$ such that 
\beq\label{eq:Maxwell}
\left\{
\begin{aligned}
\int_D \kappa (\nabla \times \bsu) \cdot (\nabla \times \bsv) dx &- \omega^2 \int_D \varepsilon \bsu \cdot \bsv dx + \int_D \nabla p \cdot \bsv dx \\
&= \int_D \bsf \cdot \bsv dx \quad \forall \bsv \in H_0(\text{curl}; D), \\
\int_D \nabla q \cdot \bsu dx &= 0 \quad \forall q \in H^1_0(D),
\end{aligned}
\right.
\eeq
where $\bsu$ is the electric field vector,  $p$ is the auxiliary variable, $\omega$ is a frequency, $\bsf = i\omega \bsj$ with current source field vector $\bsj$,  $\kappa > 0$ denotes the (inverse)  magnetic permeability, $\varepsilon > 0$ denotes the electric permittivity. 
Here we only consider $\kappa$ as a varying parameter and fix $\varepsilon$ for simplicity. 
By defining the bilinear forms 
\beq
\begin{split}
a(\bsw, \bsv; \kappa ) &:= \int_D \kappa (\nabla \times \bsu) \cdot (\nabla \times \bsv) dx - \omega^2 \int_D \varepsilon \bsu \cdot \bsv dx, \quad \forall \bsw, \bsv \in \cV, \\
b(\bsv, q) &:=  \int_D \nabla p \cdot \bsv dx, \quad \forall \bsv \in \cV, \forall q \in \cQ,
\end{split}
\eeq
in the Hilbert spaces $\cV = H_0(\text{curl}; D)$ and $\cQ = H^1_0(D)$, we can express the time harmonic Maxwell system as in the abstract saddle point formulation \cite[Eq. 1]{ChenGhattas2018} with  $\cK = L^\infty(D)$. Moreover, we can verify \cite[Assumption 1]{ChenGhattas2018} with the following constants 
\beq
\gamma = \esssup_{x\in D} \kappa(x), \quad \delta = 1, \text{ and } \beta = \frac{1}{\sqrt{1+C_p}},
\eeq
and 
\beq\label{eq:alphaMaxwell}
\alpha = \frac{1}{1+C_f} \left(\essinf_{x\in D} \kappa(x) - \omega^2 C_f  \esssup_{x\in D} \varepsilon(x) \right).
\eeq
It is straightforward to verify $\gamma$ and $\delta$.  
For any $q \in \cQ$, by taking $\bsv = \nabla q$, we have 
\beq
\sup_{\bsv \in \cV} \frac{b(\bsv,q)}{||\bsv||_\cV||q||_\cQ} \geq \frac{||\nabla q||^2_{(L^2(D))^d}}{||\nabla q||_{(L^2(D))^d}||q||_{H^1_0(D)}} = \frac{||\nabla q||_{(L^2(D))^d}}{||q||_{H^1_0(D)}}\geq \frac{1}{\sqrt{1+C_p}} =: \beta,
\eeq
where we used $\nabla \times \nabla q = 0$, $\forall q \in \cQ$, in the first inequality. 
By Friedrichs' inequality \cite{BoffiBrezziFortin2013}, there exists a constant $C_f$ such that 
\beq
\int_D |\bsv|^2 dx \leq C_f \int_{D} |\nabla \times \bsv |^2 dx, \quad \forall \bsv \in \cV.
\eeq
Therefore, we have 
\beq
\begin{split}
a(\bsv, \bsv; \kappa) &\geq \essinf_{x\in D} \kappa(x) \int_{D} |\nabla \times \bsv|^2 dx - \omega^2 \esssup_{x\in D} \varepsilon(x) \int_D|\bsv|^2 dx \\
& \geq \left(\essinf_{x\in D} \kappa(x) - \omega^2 C_f \esssup_{x\in D} \varepsilon(x)  \right) \int_{D} |\nabla \times \bsv|^2 dx \\
& \geq \frac{1}{1+C_f} \left(\essinf_{x\in D} \kappa(x) - \omega^2 C_f \esssup_{x\in D} \varepsilon(x)  \right) ||\bsv||_\cV^2, \quad \forall \bsv \in \cV.
\end{split}
\eeq
%
